\documentclass[11pt, a4paper]{article}

\usepackage[utf8]{inputenc}

\usepackage{amsmath, amsfonts, amssymb, amsthm, cases, csquotes, enumitem, hyperref, mathrsfs, mathtools, tikz, tikz-cd}
\usepackage[style=alphabetic]{biblatex}
\theoremstyle{definition}

\newtheorem{definition}{Definition}[section]
\newtheorem{theorem}{Theorem}[section]
\newtheorem{proposition}{Proposition}[section]
\newtheorem{corollary}{Corollary}[section]
\newtheorem{lemma}{Lemma}[section]
\newtheorem{example}{Example}[section]
\newtheorem{conjecture}{Conjecture}[section]

\usepackage[left=2cm,right=2cm,top=2cm,bottom=2cm,bindingoffset=0cm]{geometry}

\title{The Arens-Michael envelopes of Laurent Ore extensions}
\author{Petr Kosenko}

\newcommand{\Addresses}{{
		\bigskip
		\footnotesize
		
		P.~Kosenko,	\textit{E-mail:} \texttt{prkosenko@edu.hse.ru}, \texttt{petr.kosenko@mail.utoronto.ca}
}}

\newcommand{\verteq}{\rotatebox{90}{$\,=$}}
\newcommand{\equalto}[2]{\underset{\scriptstyle\overset{\mkern4mu\verteq}{#2}}{#1}}

\begin{document}
\maketitle

\begin{abstract}
For an Arens-Michael algebra $A$ we consider a class of $A$-$\hat{\otimes}$-bimodules which are invertible with respect to the projective bimodule tensor product. We call such bimodules topologically invertible over $A$. Given a Fr\'echet-Arens-Michael algebra $A$ and an topologically invertible Fr\'echet $A$-$\hat{\otimes}$-bimodule $M$, we construct an Arens-Michael algebra $\widehat{L}_A(M)$ which serves as a topological version of the Laurent tensor algebra $L_A(M)$.

Also, for a fixed algebra $B$ we provide a condition on an invertible $B$-bimodule $N$ sufficient for the Arens-Michael envelope of $L_B(N)$ to be isomorphic to $\widehat{L}_{\widehat{B}}(\widehat{N})$. In particular, we prove that the Arens-Michael envelope of an invertible Ore extension $A[x, x^{-1}; \alpha]$ is isomorphic to $\widehat{L}_{\widehat{A}}(\widehat{A}_{\widehat{\alpha}})$ provided that the Arens-Michael envelope of $A$ is metrizable.
\end{abstract}

\section*{Introduction}

We'd like to begin the paper by demonstrating the connection between the Arens-Michael envelopes and noncommutative geometry.

Noncommutative geometry is a branch of mathematics which, in particular, arose from such fundamental results as the first Gelfand-Naimark theorem or the Nullstellensatz theorem, or, more precisely, their categorical interpretations:

\begin{theorem}[the first Gelfand-Naimark theorem]
	Denote the category of commutative unital $C^*$-algebras by $\mathbf{CUC}^*$ and the category of compact Hausdorff topological spaces by $\mathbf{Comp}$. Then a pair of functors $\mathcal{F} : \mathbf{Comp} \rightarrow \mathbf{CUC}^*$ and $\mathcal{G} : \mathbf{CUC}^* \rightarrow \mathbf{Comp}$, where $\mathcal{F}(X) = C(X)$ and $\mathcal{G}(A) = \text{Spec}_m(A)$, is an anti-equivalence of categories.
\end{theorem}

\begin{theorem}[Nullstellensatz]
	Let $\mathbb{K}$ be an algebraically closed field. Denote the category of affine algebraic $\mathbb{K}$-varieties by $\mathbf{Aff}$ and the category of commutative finitely generated reduced unital $\mathbb{K}$-algebras by $\mathbf{Alg}$. Then a pair of functors $\mathcal{F} : \mathbf{Aff} \rightarrow \mathbf{Alg}$ and $\mathcal{G} : \mathbf{Alg} \rightarrow \mathbf{Aff}$, where $\mathcal{F}(X) = \mathbb{K}[X]$ and $\mathcal{G}(A) = \text{Spec}_m(A)$, is an anti-equivalence of categories.
\end{theorem}

As the reader can see, these theorems state that some category of geometrical objects is anti-equivalent to the category of functions on them. This observation, for example, serves as a motivation to think of noncommutative $C^*$-algebras as the function spaces of ``noncommutative topological spaces''.

The notion of Arens-Michael envelopes was discovered by J. Taylor in \cite{taylor1972general} due to the problem of multi-operator functional calculus existence. It is worth noting that the terminology the author used was different from that we use nowadays: Taylor defined them as ``completed locally submultiplicative convex envelopes''. The current terminology is due to A. Helemskii, see \cite{Helem1993}.

The following theorem serves as a motivation to study Arens-Michael envelopes in the context of noncommutative geometry.

\begin{theorem}
	\label{holomorphicfunctions}
	The Arens-Michael envelope of $\mathbb{C}[t_1, \dots, t_n]$ is topologically isomorphic to the algebra of holomorphic functions $\mathcal{O}(\mathbb{C}^n)$ endowed with the compact-open topology.
\end{theorem}

This theorem, attributed to J.Taylor, can be formulated as follows: the Arens-Michael envelope of the algebra of regular functions on $\mathbb{C}^n$ is isomorphic to the algebra of holomorphic functions on $\mathbb{C}^n$. In fact, the same holds true for an arbitrary affine complex algebraic variety, see \cite[Example 3.6]{PirkAM}. Therefore, it makes sense to define the algebra of ``holomorphic functions'' on a noncommutative affine algebraic variety as the Arens-Michael envelope of the algebra of ``regular functions'' on it. In other words, the notion of Arens-Michael envelope serves as a ``bridge'' between algebra and functional analysis.

%
%


In this paper we are concerned with ``computing'' the Arens-Michael envelopes for some interesting non-commutative associative finitely generated over $\mathbb{C}$ algebras. By this we mean that for an algebra $A$ we aim to explicitly construct the Arens-Michael algebra $B$ which turns out to be isomorphic to the Arens-Michael envelope $\widehat{A}$. In most (non-degenerate) cases such algebra $B$ is constructed as a power series algebra. In other words,  the underlying locally convex space of $B$ turns out to be a K\"{o}the space.

One of the more effective approaches which we use to compute Arens-Michael envelopes lies in considering Ore extensions. Suppose that $A$ is an algebra with an endomorphism $\alpha \in \text{End}(A)$ and an $\alpha$-derivation $\delta: A \rightarrow A$. Then, under some reasonable conditions on $\alpha$ and $\delta$, the Arens-Michael envelope of $A[t; \alpha, \delta]$ admits a description in terms of the Arens-Michael envelope $\widehat{A}$.

A lot of naturally occurring noncommutative algebras can be represented as iterated Ore extensions, for example, $q$-deformations of classical algebras, such as $\text{Mat}_q(2)$ or $U_q(\mathfrak{g})$.

Let $A$ be a unital associative complex algebra, $\alpha \in \text{End}(A)$ and $\delta : A \rightarrow A$ be an $\alpha$-derivation. Consider the Ore extension $A[t; \alpha, \delta]$. Then there are several cases:

\begin{enumerate}
	\item
	Suppose that the pair $(\alpha, \delta)$ is ``nice enough'' in the sense that their extensions to $\widehat{A}$ behave well enough with respect to the topology on $\widehat{A}$ ($m$-localizable families of morphisms). Then the Arens-Michael envelope of $A[t; \alpha, \delta]$ admits a relatively simple description, see \cite[Proposition 4.5]{PirkAM} and \cite[Theorem 5.17]{PirkAM}.\\
	Now suppose that $\delta = 0$ for simplicity.
	\item
	Consider now the case of general $\alpha \in \text{End}(A)$. Then $A[t; \alpha]$ still admits the explicit description of the Arens-Michael envelope, which utilizes the analytic version of the notion of tensor algebra, associated with a bimodule. However, the seminorms, which were used to describe the topology on $A\{t; \alpha\} \simeq \widehat{A[t; \alpha]}$, are difficult to compute in the general case, see \cite[Proposition 4.9]{PirkAM}, \cite[Corollary 5.6]{PirkAM} and \cite[Example 4.3]{PirkAM}.
	\item
	Now, suppose that $\alpha$ is invertible. Then one can define the Laurent Ore extension $A[t, t^{-1}; \alpha]$.
	Again, there is a case, where the pair $(\alpha, \alpha^{-1})$ is ''nice enough'', then the Arens-Michael envelope admits a description similar to the one in the case 1, see \cite[Proposition 4.15]{PirkAM} and \cite[Theorem 5.21]{PirkAM}.
	\item
	The case of arbitrary $\alpha \in \text{Aut}(A)$ is treated by the author in this paper, it is worth mentioning that the approach is inspired by methods, used in \cite{PirkAM}. In particular, in this article we introduce the analytic version of the Laurent tensor algebra, associated with an topologically invertible bimodule. 
	\item
	The most general case, $A[t; \alpha, \delta]$, is still out of reach, unfortunately.
\end{enumerate}

The structure of this paper is as follows: in Section \ref{Definitions} we give the definitions of different types of topological algebras and their Arens-Michael envelopes. Subsections \ref{Some algebraic constructions} and \ref{Invertible bimodules and the Laurent tensor algebra} are devoted to defining several algebraic constructions, in particular, the Laurent tensor algebra of an invertible bimodule $L_A(M)$. In the latter subsections we introduce the analytic analogue of $L_A(M)$, formulate and prove one of the main results of the paper.

In Section \ref{special case} we tackle the special case $M = A_\alpha$ and describe $\widehat{L}_A(A_\alpha)$ as explicitly as possible for any Fr\'echet-Arens-Michael algebra $A$ and continuous automorphism $\alpha : A \rightarrow A$. 
In the Section \ref{open questions} we state some open problems related to the Arens-Michael envelopes. We also provide some examples in Appendix \ref{appendix A}.

The main results of this paper are contained in the Theorems \ref{AMofinvertible}, \ref{topolaurent}, \ref{thm3.1}, and the Corollaries \ref{firstisom}, \ref{secondisom}.

This paper started as a translation of the author's bachelor paper ``Arens-Michael envelopes of some associative algebras'', which was written in Russian (not published).

\subsection*{Acknowledgments}
I am extremely grateful to Alexei Yu. Pirkovskii, without whom this paper would not exist, and who provided many helpful comments on earlier versions of this paper. Also I would like to thank the participants of the conference "Banach Algebras and Applications" who attended my talk about the results, which were obtained in this paper.

\section{Definitions}
\label{Definitions}

\subsection{Basic notions}
As in \cite{PirkAM}, by an \textit{algebra} we mean an associative unital $\mathbb{C}$-algebra.

For a detailed introduction to the theory of locally convex spaces and algebras the reader can see \cite{book:29603}, \cite{book:955359}, \cite{book:338950} or \cite{Helem2006}.

\begin{definition}
	Let $A$ be a locally convex space with a multiplication $\mu : A \times A \rightarrow A$, such that $(A, \mu)$ is an algebra.
	\begin{enumerate}[label=(\arabic*)]
		\item If $\mu$ is separately continuous, then $A$ is called \textit{a locally convex algebra}. 
		\item If $A$ is a complete locally convex space, and $\mu$ is jointly continuous, then $A$ is called a $\hat{\otimes}$-\textit{algebra}.
	\end{enumerate}
\end{definition}

\begin{definition}
	A locally convex algebra $A$ is called $m$-convex if the topology on it can be defined by a family of submultiplicative seminorms.
\end{definition}

\begin{definition}
	A complete locally $m$-convex algebra is called an \textit{Arens-Michael algebra}.
\end{definition}
\noindent
For us it will be important to keep in mind the following examples of Arens-Michael algebras:
\begin{enumerate}
	\item  Any Banach algebra is an Arens-Michael algebra.
	\item  For any $n \in \mathbb{N}$ the algebra $\mathcal{O}(\mathbb{C}^n)$ of holomorphic functions on $\mathbb{C}^n$, endowed with the compact-open topology, is an Arens-Michael algebra.
	\item  For any locally compact space $X$ the algebra of continuous functions $C(X)$, endowed with the topology of uniform convergence on compact sets of $X$, is an Arens-Michael algebra.
\end{enumerate}

Also, keep in mind that every Arens-Michael algebra is a $\hat{\otimes}$-algebra.

\begin{definition}
	Let $A$ be a $\hat{\otimes}$-algebra and let $M$ be a complete locally convex space with a (purely algebraic) structure of an $A$-module. Suppose that the natural maps $A \times M \rightarrow M$ and $M \times A \rightarrow M$ are jointly continuous. Then $M$ is called a $A$-$\hat{\otimes}$\textit{-bimodule}.
\end{definition}

\subsection{Arens-Michael envelopes}

\begin{definition}
	Let $A$ be an algebra. An \textit{Arens-Michael envelope} of $A$ is a pair $(\widehat{A}, i_A)$, where $\widehat{A}$ is an Arens-Michael algebra and $i_A : A \rightarrow \widehat{A}$ is an algebra homomorphism, satisfying the following universal property: for any Arens-Michael algebra $B$ and algebra homomorphism $\varphi : A \rightarrow B$ there exists a unique continuous algebra homomorphism $\hat{\varphi} : \widehat{A} \rightarrow B$ extending $\varphi$, i.e. $\varphi = \hat{\varphi} \circ i_A$:
	\[
	\begin{tikzcd}
	\widehat{A} \ar[r, "\hat{\varphi}", dashed] & B \\
	A \ar[u, "i_A"] \ar[ur, "\varphi"'] & 
	\end{tikzcd}
	\]
\end{definition}


The Arens-Michael envelope of an algebra always exists and is unique up to a topological isomorphism, it is isomorphic to the completion of $A$ with respect to the family of all submultiplicative seminorms on $A$.

We have already mentioned the Theorem \ref{holomorphicfunctions}, which serves as a fundamental example of a computation of the Arens-Michael envelope. Here are some more important examples, which we borrow from \cite{PirkAM}:

\begin{example}
	\label{freeexample}
	Denote the free algebra with generators $\xi_1, \dots, \xi_n$ over $\mathbb{C}$ by $F_n$. Then its Arens-Michael envelope is a locally convex algebra, which will be denoted by $\mathcal{F}_n$, defined as follows:
	\[
	\mathcal{F}_n := \left\lbrace a = \sum_{w \in W_n} a_w \xi^w : \left\| a \right\| _\rho = \sum_{w \in W_n} |a_w| \rho^{|w|} < \infty \ \forall \, 0 < \rho < \infty \right\rbrace.
	\]
	In particular, $\mathcal{F}_n$ is a nuclear Fr\'echet algebra.
\end{example}

\begin{example}
	Let $\mathfrak{g}$ be a finite-dimensional semisimple Lie algebra. The Arens-Michael envelope of $U(\mathfrak{g})$ is isomorphic to the direct product $\prod\limits_{V \in \hat{\mathfrak{g}}} \text{Mat}(V)$, where $\hat{\mathfrak{g}}$ is the set of the equivalence classes of the finite-dimensional irreducible reps of $\mathfrak{g}$.
\end{example}
Sometimes the Arens-Michael envelope of an algebra is isomorphic to the zero algebra:
\begin{example}
	Suppose that $A$ is an algebra generated by $x$ and $y$ with the single relation $xy - yx = 1$. Then $\widehat{A} = 0$, because an arbitrary non-zero Banach algebra $B$ cannot contain such elements $x, y \in B$, such that $[x, y] = 1$.
\end{example}
The definition of Arens-Michael envelopes can be given in case of bimodules, too.
\begin{definition}
	Let $A$ be a algebra and suppose that $M$ is an $A$-bimodule. Then an \textit{Arens-Michael envelope} of $M$ is a pair $(\widehat{M}, i_M)$, where $\widehat{M}$ is a $\widehat{A}$-$\hat{\otimes}$-bimodule and $i_M : M \rightarrow \widehat{M}$ is a $A$-bimodule homomorphism, which satisfies the following universal property: for any $\widehat{A}$-$\hat{\otimes}$-bimodule $N$ and $A$-bimodule homomorphism $f : M \rightarrow N$ there exists a unique continuous $\widehat{A}$-$\hat{\otimes}$-bimodule homomorphism $\hat{f} : \widehat{M} \rightarrow N$ which extends $f$:
	\[
	\begin{tikzcd}
	\widehat{M} \ar[r, "\hat{f}", dashed] & N \\
	M \ar[u, "i_A"] \ar[ur, "f"'] & 
	\end{tikzcd}
	\]
\end{definition}


In this paper we will use \cite[Proposition 6.1]{PirkStfl}, which, basically, states that the Arens-Michael functor commutes with quotients:

\begin{theorem}
	Suppose that $A$ is an algebra and $I \subset A$ is a two-sided ideal. Denote by $J$ the closure of $i_A(I)$ in $\widehat{A}$. Then $J$ is a closed two-sided ideal in $\widehat{A}$ and the induced homomorphism $A/I \rightarrow \widehat{A}/J$ extends to a topological algebra isomorphism
	\[
	\widehat{(A/I)} \simeq (\widehat{A} / J)^\sim,
	\]
	where $\widetilde{A}$ denotes the completion of $A$ as a locally convex space.
	
	Moreover, if $\widehat{A}$ is a Fr\'echet algebra, then we do not need to complete the quotient, so we have
	\[
	\widehat{(A/I)} \simeq \widehat{A} / J.
	\]
\end{theorem}
\textbf{Remark.} As a corollary from this theorem and the Example \ref{freeexample} we have that the Arens-Michael algebra of any finitely generated algebra over $\mathbb{C}$ is a Fr\'echet algebra.

\section{Topological analogues of invertible bimodules their Laurent tensor algebras}
\label{Defining topological analogues of invertible bimodules and constructing the analytic Laurent tensor algebras}
\subsection{Some algebraic constructions}
\label{Some algebraic constructions}

Firstly, let's recall the definitions of some crucial algebraic constructions, which we will use throughout this paper:

\begin{definition}
	Let $A$ be an algebra and suppose that $\alpha$ is an endomorphism of $A$. Then we define $A_\alpha$ as a $A$-bimodule which coincides with $A$ as a left $A$-module and $x \circ a = x \alpha(a)$ for $x \in A_\alpha$, $a \in A$. Similarly, one defines ${}_\alpha A$.
\end{definition}

\begin{definition}
	Let $A$ be an algebra, $\alpha \in \text{End}(A)$ and $\delta \in \text{Der}(A, {}_\alpha A)$, or, equivalently,
	\[
\delta(ab) = \alpha(a) \delta(b) + \delta(a) b \quad \forall \, a, b \in A.
	\]
	Then the \textit{Ore extension} of $A$ with respect to $\alpha$ and $\delta$ is the vector space \[A[t; \alpha, \delta] = \left\lbrace  \sum_{i=0}^n a_i t^i : a_i \in A \right\rbrace  \] with the multiplication, which is uniquely defined by the following conditions:
	\begin{enumerate}
		\item The relation $ta = \alpha(a) t + \delta(a)$ holds for any $a \in A$
		\item The natural inclusions $A \hookrightarrow A[t; \alpha, \delta]$ and $\mathbb{C}[t] \hookrightarrow A[t; \alpha, \delta]$ are algebra homomorphisms.
	\end{enumerate}
	Also, if $\delta = 0$ and $\alpha$ is invertible, then one can define the \textit{Laurent Ore extension} \[A[t, t^{-1}; \alpha] = \left\lbrace \sum_{i = -n}^n a_i t^i : a_i \in A \right\rbrace \] with the multiplication defined in a similar way.
\end{definition}

\subsection{Invertible bimodules and the Laurent tensor algebra}
\label{Invertible bimodules and the Laurent tensor algebra}
\textbf{Remark.} We were not able to find any references about Laurent tensor algebras of invertible modules in the literature, so we decided to provide a basic exposition in this section.
\begin{definition}
	\label{invdef}
	Suppose that $A$ is an algebra and $M$ is an $A$-bimodule. Then $M$ is called an \textit{invertible} $A$\textit{-bimodule} if there exist an $A$-bimodule $M^{-1}$ together with $A$-bimodule isomorphisms $i_1 : M \otimes_A M^{-1} \simeq A$ and $i_2 : M^{-1} \otimes_A M \simeq A$ (which we shall call convolutions) such that the following diagrams commute:
	\begin{equation}
	\label{associativity diagrams}
		\begin{tikzcd}
			M \otimes_{A} M^{-1} \otimes_{A} M \ar[r, "Id_M \otimes i_2"] \ar[d, "i_1 \otimes Id_M"] & M \otimes_{A} A \ar[d, "m\otimes a \rightarrow ma"] & M^{-1} \otimes_{A} M \otimes_{A} M^{-1} \ar[r, "Id_M \otimes i_1"] \ar[d, "i_2 \otimes Id_M"] & M^{-1} \otimes_{A} A \ar[d, "n\otimes a \rightarrow na"] \\ 
			A \otimes_{A} M \ar[r, "a\otimes m \rightarrow am"] & M & A \otimes_{A} M^{-1} \ar[r, "a\otimes n \rightarrow an"] & M^{-1}
		\end{tikzcd}
	\end{equation}
\end{definition}

With any $A$-bimodule $M$ one associates the tensor algebra $T_A(M)$:
\begin{center}
	$T_A(M) := A \oplus \bigoplus\limits_{n \in \mathbb{N}} M^{\otimes n}$
\end{center}
In turn, for every invertible $A$-bimodule we can define a complex vector space which will be denoted by $L_A(M)$:
\begin{equation}
	L_A(M) := \bigoplus\limits_{n \in \mathbb{Z}} M^{\otimes n},
\end{equation}
where $M^{\otimes-n} := (M^{-1})^{\otimes n}$ and $M^{\otimes 0} := A$.

The elements belonging to $M^{\otimes n}$ for some $n \in \mathbb{Z}$ will be called \textit{homogeneous} of degree $n$. The following proposition states that $L_A(M)$ admits a natural algebra structure:

\begin{proposition}
	Suppose that $A$ is an algebra and $M$ is an invertible $A$-bimodule. Then $L_A(M)$ admits a unique multiplication that makes it into an associative algebra and satisfies the following conditions:
	\begin{enumerate}[label=(\arabic*)]
		\item the natural inclusions
		$j_M : T_A(M) \rightarrow L_A(M)$ and $j_{M^{-1}} : T_A(M^{-1}) \rightarrow L_A(M)$ are algebra homomorphisms.
		\item for any $m \in M$ and $n \in M^{-1}$ we have $m \cdot n = i_1(m \otimes n)$ and $n \cdot m = i_2(n \otimes m)$.
	\end{enumerate}
\end{proposition}

\begin{proof}
	It suffices to define the multiplication on the homogeneous elements of $L_A(M)$. Fix \newline ${m_1 \otimes \dots \otimes m_k \in M \otimes \dots \otimes M}$ and $n_1 \otimes \dots \otimes n_l \in M^{-1} \otimes \dots \otimes M^{-1}$. Then define \[(m_1 \otimes \dots m_k) \cdot (n_1 \otimes \dots \otimes n_l) = (m_1 \otimes \dots m_{k-1} i_1(m_k \otimes n_1)) \cdot (n_2 \otimes \dots \otimes n_l)\] and \[(n_1 \otimes \dots \otimes n_l) \cdot (m_1 \otimes \dots m_k) = (n_1 \otimes \dots n_{l-1} i_2(n_l \otimes m_1)) \cdot (m_2 \otimes m_k),\] then we repeat the process until we get a homogeneous element of $L_A(M)$. The associativity of the resulting algebra is a straightforward corollary from the commutativity of \eqref{associativity diagrams} in the Definition \ref{invdef}. 
\end{proof}

Let us call the resulting algebra $L_A(M)$ the \textit{Laurent tensor algebra} of an invertible bimodule $M$.

The following proposition immediately follows from the constructions of $T_A(M)$ and $L_A(M)$.

\begin{proposition}
	\label{twisted bimodule ex}
	Suppose that $A$ is an algebra and $\alpha$ is an automorphism of $A$.
	\begin{enumerate}[label=(\arabic*)]
		\item $A_\alpha$ and $A_{\alpha^{-1}}$ are inverse $A$-bimodules with respect to the maps 
		\[ 
		i_1(a \otimes b) = a \alpha(b), \quad i_2(b \otimes a) = b \alpha^{-1}(a). 
		\]
		\item Moreover, $T(A_\alpha) \simeq A[t; \alpha]$, $L(A_\alpha) \simeq A[t, t^{-1}; \alpha]$.
	\end{enumerate}
\end{proposition}

The algebra $L_A(M)$ satisfies the following universal property:

\begin{definition}
	Let $A$ be an algebra and consider an $A$-algebra $B$ with respect to a homomorphism $\theta : A \rightarrow B$ together with $A$-bimodule homomorphisms $\alpha : M \rightarrow B$, $\beta : M^{-1} \rightarrow B$. Then we will call the triple of morphisms $(\theta, \alpha, \beta, B)$ \textit{compatible} if and only if the following diagram is commutative:
	\begin{equation}
	\label{6diag}
	\begin{tikzcd}
	M \otimes_{A} M^{-1} \ar[r, "i_1"] \ar[d, "\alpha \otimes \beta"] & A \ar[d, "\theta"] & M^{-1} \otimes_{A} M \ar[d, "\beta \otimes \alpha"] \ar[l, "i_2"'] \\
	B \otimes_{A} B \ar[r, "m"] & B & B \otimes_A B \ar[l, "m"']
	\end{tikzcd}
	\end{equation} 
\end{definition}

\begin{proposition}
	The triple of morphisms $(i_A, i_M, i_{M^{-1}}, L_A(M))$, where all morphisms are tautological inclusions into $L_A(M)$, is a universal compatible triple, i.e. for any other algebra $B$ and any compatible triple of morphisms $(\theta, \alpha, \beta, B)$ there exists a unique $A$-algebra homomorphism $f : L_A(M) \rightarrow B$ such that the following diagrams commute:
	\begin{equation}
		\label{triangular diagrams}
		\begin{tikzcd}
			L_A(M) \ar[r, "f"] & B & L_A(M) \ar[r, "f"] & B & L_A(M) \ar[r, "f"] & B \\
			A \ar[u, "i_A"] \ar[ur, "\theta"'] & & M \ar[u, "i_M"] \ar[ur, "\alpha"'] & & M^{-1} \ar[u, "i_{M^{-1}}"] \ar[ur, "\beta"'] & 
		\end{tikzcd}
	\end{equation}
\end{proposition}
\begin{proof}
	It suffices to check the existence, as the uniqueness will follow as a standard category-theoretic argument.
	
	And the existence is straightforward: for every $m_1 \otimes \dots \otimes m_k \in M^{\otimes k}$ define
	\[
	f(m_1 \otimes \dots \otimes m_k) = \alpha(m_1) \dots \alpha(m_k),
	\]
	and for $n_1 \otimes \dots \otimes n_l \in M^{\otimes -l}$ we define
	\[
	f(n_1 \otimes \dots \otimes n_l) = \beta(n_1) \dots \beta(n_l),
	\]
	and $f(a) = \theta(a)$ for every $a \in A$.
	
	The commutativity of \eqref{6diag} ensures that $f$ is a well-defined homomorphism of $A$-algebras. And the diagrams \eqref{triangular diagrams} commute due to the construction of $f$.
\end{proof}

\subsection{Topologically invertible bimodules}
Now, using the language of locally convex vector spaces, we will construct topological versions of the notions we described above.

\begin{definition}
	Let $A$ be a $\hat{\otimes}$-algebra and $M$ be an $A$-$\hat{\otimes}$-bimodule. Then we will call $M$ a \textit{topologically invertible} $A$-$\hat{\otimes}$\textit{-bimodule} if there exist an $A$-$\hat{\otimes}$-bimodule $M^{-1}$ and two $A$-$\hat{\otimes}$-bimodule topological isomorphisms $i_1 : M \hat{\otimes}_A M^{-1} \simeq A$ and $i_2 : M^{-1} \hat{\otimes}_A M \simeq A$, such that the following diagrams commute:
	\begin{equation}
		\label{topological associativity diagrams}
			\begin{tikzcd}
				M \hat{\otimes}_{A} M^{-1} \hat{\otimes}_{A} M \ar[r, "Id_M \hat{\otimes} i_2"] \ar[d, "i_1 \hat{\otimes} Id_M"] & M \hat{\otimes}_{A} A \ar[d, "m \otimes r \rightarrow mr"] & M^{-1} \hat{\otimes}_{A} M \hat{\otimes}_{A} M^{-1} \ar[r, "Id_M \hat{\otimes} i_1"] \ar[d, "i_2 \hat{\otimes} Id_M"] & M^{-1} \hat{\otimes}_{A} A \ar[d, "n \otimes r \rightarrow nr"] \\ 
				A \hat{\otimes}_{A} M \ar[r, "r \otimes m \rightarrow rm"] & M & A \hat{\otimes}_{A} M^{-1} \ar[r, "r \otimes n \rightarrow rn"] & M^{-1}
			\end{tikzcd}
	\end{equation}
\end{definition}

The following proposition is the topological version of the Proposition \ref{twisted bimodule ex}.
\begin{proposition}
	\label{top. twisted bimodule}
	Let $A$ be a $\hat{\otimes}$-algebra and suppose that $\alpha$ is an automorphism of $A$. Then $A_\alpha$ and $A_{\alpha^{-1}}$ are topologically inverse $A$-$\hat{\otimes}$-bimodule with respect to the maps
	\[
	i_1(a \otimes b) = a \alpha(b), \quad i_2(b \otimes a) = b \alpha^{-1}(a).
	\]
\end{proposition}
More information on topologically invertible bimodules can be found in \cite{PirkVdB}.

There is a natural question related to Arens-Michael envelopes: is it true that the Arens-Michael envelope of an invertible bimodule is topologically invertible? At the moment we can state a conjecture:
\begin{conjecture}
	\label{conj1}
	Suppose that $A$ is an algebra and $M$ is an invertible $A$-bimodule. 
	Then there exist topological $A$-$\hat{\otimes}$-bimodule isomorphisms $\hat{i}_1 : \widehat{M} \hat{\otimes}_{\widehat{A}} \widehat{M^{-1}} \rightarrow \widehat{A}$ and $\hat{i}_2 : \widehat{M^{-1}} \hat{\otimes}_{\widehat{A}} \widehat{M} \rightarrow \widehat{A}$, satisfying the following conditions:
	\begin{enumerate}[label=(\arabic*)]
		\item $\widehat{M}$ is an topologically invertible $\widehat{A}$-$\hat{\otimes}$-bimodule w.r.t. $i_1$ and $i_2$.
		\item The following diagram is commutative:
		\begin{equation}
		\label{compat}
		\begin{tikzcd}
		M \otimes_{A} M^{-1} \ar[r, "i_1"] \ar[d, "i_M \otimes i_{M^{-1}}"] & A \ar[d, "i_A"] & M^{-1} \otimes_{A} M \ar[d, "i_{M^{-1}} \otimes i_M"] \ar[l, "i_2"'] \\
		\widehat{M} \hat{\otimes}_{\widehat{A}} \widehat{M^{-1}} \ar[r, "\hat{i}_1"] & \widehat{A} & \widehat{M^{-1}} \hat{\otimes}_{\widehat{A}} \widehat{M} \ar[l, "\hat{i}_2"']
		\end{tikzcd}
		\end{equation}
		where the left arrow maps $a \otimes b$ to $i_M(a) \otimes i_{M^{-1}}(b)$, and the right arrow maps $b \otimes a$ to $i_{M^{-1}}(b) \otimes i_M(a)$.
	\end{enumerate}
\end{conjecture}

It turns out that there is a particular case in which, at least, the first statement of the above conjecture holds.

\begin{theorem}
	\label{AMofinvertible}
	Consider a $\hat{\otimes}$-algebra $A$ and a pair of topologically invertible $A$-$\hat{\otimes}$-bimodules $M, M^{-1}$. Suppose that the following conditions are satisfied:
	\begin{enumerate}[label=(\arabic*)]
		\item $i_A : A \rightarrow \widehat{A}$ is an epimorphism in the category of $\hat{\otimes}$-algebras. Equivalently, 
		\[
		\mu : \widehat{A} \hat{\otimes}_A \widehat{A} \xrightarrow{\sim} \widehat{A}, \quad (a, b) \rightarrow ab
		\]
		is a  $\widehat{A}$-$\hat{\otimes}$-bimodule isomorphism.
		\item $M \hat{\otimes}_A \widehat{A} \simeq \widehat{A} \hat{\otimes}_A M$, $M^{-1} \hat{\otimes}_A \widehat{A} \simeq \widehat{A} \hat{\otimes}_A M^{-1}$ as $A$-$\hat{\otimes}$-bimodules.
	\end{enumerate}
	Then $\widehat{M}$ and $\widehat{M}^{-1}$ are topologically invertible $\widehat{A}$-bimodules.
\end{theorem}
\textbf{Remark.} Here we consider the Arens-Michael envelopes of topological algebras and bimodules, see \cite[Section 3]{PirkAM} for the details.

\begin{proof}
	It is an immediate corollary of the fact that $\widehat{M} \simeq \widehat{A} \hat{\otimes}_A M \hat{\otimes}_A \widehat{A}$, which is the statement of \cite[Remark 3.8]{PirkAM}, and \cite[Proposition 10.4]{PirkVdB}. The idea is to write the following chain of $\widehat{A}$-$\hat{\otimes}$-bimodule isomorphisms:
	\begin{equation}
		\begin{aligned}
			\widehat{M} & \hat{\otimes}_{\widehat{A}} \widehat{M^{-1}} \stackrel{R3.8}{\simeq} \widehat{A} \hat{\otimes}_A M  \hat{\otimes}_A \widehat{A} \hat{\otimes}_{\widehat{A}} \widehat{A} \hat{\otimes}_A M^{-1} \hat{\otimes}_A \widehat{A} \simeq \widehat{A} \hat{\otimes}_A M  \hat{\otimes}_A \widehat{A} \hat{\otimes}_A M^{-1} \hat{\otimes}_A \widehat{A} \stackrel{\text{P10.4 (ii)}}{\simeq} \\ & \simeq \widehat{A} \hat{\otimes}_A M  \hat{\otimes}_A M^{-1} \hat{\otimes}_A \widehat{A} \hat{\otimes}_{A} \widehat{A} \stackrel{(1)}{\simeq} \widehat{A} \hat{\otimes}_A M  \hat{\otimes}_A M^{-1} \hat{\otimes}_A \widehat{A}  \xrightarrow{i_1} \widehat{A} \hat{\otimes}_A A \hat{\otimes}_A \widehat{A} \stackrel{R3.8}{\simeq} \widehat{A}.
		\end{aligned}
	\end{equation}
	In a similar fashion we can show that the associativity diagrams commute.
\end{proof}

As a corollary, consider an algebra $A$ and a pair of invertible bimodules $M$,  $M^{-1}$ of at most countable dimension. Then \cite[Proposition 2.3]{PirkAM} implies the following statements:
\begin{enumerate}[label=(\arabic*)]
	\item $A_s$ is a $\hat{\otimes}$-algebra, $M_s$, $(M^{-1})_s$ are $A_s$-$\hat{\otimes}$-bimodules.
	\item These bimodules are topologically invertible as $A_s$-bimodules.
\end{enumerate}
Suppose that $A_s$, $M_s$ and $(M^{-1})_s$ satisfy the conditions of the Theorem \ref{AMofinvertible}. Then $\widehat{M}$ and $\widehat{M^{-1}}$ are topologically invertible $\widehat{A}$-bimodules.

\begin{proposition}
	\label{example where conj is satisfied}
	The conjecture \ref{conj1} holds in the case of $M = A_\alpha$, where $A$ is an arbitrary associative algebra and $\alpha \in \text{Aut}(A)$.
\end{proposition}
\begin{proof}
	We refer to the \cite[Corollary 5.6]{PirkAM} which states that $\widehat{A_\alpha} \simeq (\widehat{A})_{\hat{\alpha}}$. Taking the necessary isomorphisms from the Proposition \ref{top. twisted bimodule}, we get (1), and the following computation proves the commutativity of the left side of \eqref{compat}:
	\[
		\hat{i}_1(i_A(a) \otimes i_A(b)) = i_A(a \alpha(b) ) = i_A \circ i_1(a \otimes b) \quad (a \in M, b \in M^{-1})
	\]

	A similar argument also shows that the right quadrant of the diagram \eqref{compat} is commutative too.
\end{proof}

\subsection{Topological Laurent tensor algebras}
Fix an Arens-Michael algebra $A$ and a pair of topologically inverse $A$-$\hat{\otimes}$-bimodules $M$ and $M^{-1}$.

\begin{definition}
	Let $B$ be an Arens-Michael algebra, which is an $A$-algebra with respect to a continuous homomorphism $\theta : A \rightarrow B$, also let $\alpha : M \rightarrow B$, $\beta : M^{-1} \rightarrow B$ be continuous $A$-$\hat{\otimes}$-bimodule homomorphisms. Then we will call the triple $(\theta, \alpha, \beta, B)$ \textit{topologically compatible} if and only if the following diagram is commutative:
	\begin{equation}
	\label{analyt lta}
		\begin{tikzcd}
			M \hat{\otimes}_{A} M^{-1} \ar[r, "i_1"] \ar[d, "\alpha \hat{\otimes} \beta"] & A \ar[d, "\theta"] & M^{-1} \hat{\otimes}_{A} M \ar[d, "\beta \hat{\otimes} \alpha"] \ar[l, "i_2"'] \\
			B \hat{\otimes}_{A} B \ar[r, "m"] & B & B \hat{\otimes}_R B \ar[l, "m"']
		\end{tikzcd}
	\end{equation}
\end{definition}

Now we will formulate one of the main theorems of the paper:

\begin{theorem}
	\label{topolaurent}
	Let $A$ be a Fr\'echet-Arens-Michael algebra and consider topologically inverse Fr\'echet $A$-$\hat{\otimes}$-bimodules $M$, $M^{-1}$. Then there exist an Arens-Michael algebra $\widehat{L}_R(M)$ and a topologically compatible triple of morphisms $(\theta, \alpha, \beta, \widehat{L}_R(M))$ that satisfies the following universal property: for every Arens-Michael algebra $B$ and a topologically compatible triple of morphisms $(\theta', \alpha', \beta', B)$ there exists a unique continuous $A$-algebra homomorphism $f : \widehat{L}_R(M) \rightarrow B$ such that the following diagrams commute:
	\begin{equation}
		\label{topological triangular diagrams}
		\begin{tikzcd}
			\widehat{L}_R(M) \ar[r, "f"] & B & \widehat{L}_R(M) \ar[r, "f"] & B & \widehat{L}_R(M) \ar[r, "f"] & B \\
			A \ar[u, "\theta"] \ar[ur, "\theta'"'] & & M \ar[u, "\alpha"] \ar[ur, "\alpha'"'] & & M^{-1} \ar[u, "\beta"] \ar[ur, "\beta'"'] & 
		\end{tikzcd}
	\end{equation}
	
	 If this object exists, we will call it the \textit{topological(or analytic) Laurent tensor algebra} of a $A$-$\hat{\otimes}$-bimodule $M$.
\end{theorem}

The proof of the existence of the universal object will be given in the next subsection. What we want to do now is to establish the connection between analytic Laurent tensor algebras and Arens-Michael envelopes.

\begin{proposition}
	\label{Arens-Michael env of LTA}
	Now suppose that $A$ is an algebra and $M$, $M^{-1}$ are a pair of (algebraically) inverse $A$-bimodules. Suppose that the following condition holds for $\widehat{A}$, $\widehat{M}$ and $\widehat{M^{-1}}$:
	\begin{enumerate}[label=(\arabic*)]
		\item The underlying LCS of $\widehat{A}$, $\widehat{M}$ and $\widehat{M^{-1}}$ are Fr\'echet spaces. 
		\item $\widehat{M}$ and $\widehat{M^{-1}}$ are topologically inverse as $\widehat{A}$-$\hat{\otimes}$-bimodules which satisfy the Conjecture \ref{conj1}.
	\end{enumerate}
	Then, if $(\theta, \alpha, \beta, \widehat{L}_{\widehat{A}}(\widehat{M}))$ is a resulting topologically compatible triple in the Theorem \ref{topolaurent}, then $\widehat{L_A(M)} \simeq \widehat{L}_{\widehat{A}}(\widehat{M})$.
\end{proposition}

\begin{proof}
	Firstly we need to construct an algebra homomorphism $i : L_A(M) \rightarrow \widehat{L}_{\widehat{A}}(\widehat{M})$. Consider the following morphisms: $\theta i_A : A \rightarrow \widehat{A} \rightarrow \widehat{L}_{\widehat{A}}(\widehat{M})$, $\alpha i_M : M \rightarrow \widehat{M} \rightarrow \widehat{L}_{\widehat{A}}(\widehat{M})$ and $\beta i_{M^{-1}} : M^{-1} \rightarrow \widehat{M^{-1}} \rightarrow \widehat{L}_{\widehat{A}}(\widehat{M})$. It turns out that this triple of morphisms is (algebraically) compatible, however, this statement is not as obvious as one might think: look at the diagram, commutativity of which we aim to prove:
	\begin{equation}
		\begin{tikzcd}
				M \otimes_A M^{-1} \arrow[d, "\alpha i_M \otimes \beta i_{M^{-1}}"] \ar[r, "i_1"] &
				A \arrow[d, "\theta i_A"] &
				M^{-1} \otimes_A M \arrow[d, "\beta i_{M^{-1}} \otimes \alpha i_M"] \arrow[l, "i_2"'] \\
				\widehat{L}_{\widehat{A}}(\widehat{M}) \otimes_{A} \widehat{L}_{\widehat{A}}(\widehat{M}) \arrow[r, "m"] &
				\widehat{L}_{\widehat{A}}(\widehat{M}) &
				\widehat{L}_{\widehat{A}}(\widehat{M}) \otimes_{A} \widehat{L}_{\widehat{A}}(\widehat{M}) \arrow[l, "m"']
		\end{tikzcd}
	\end{equation}
	 Notice that we deal with the algebraic tensor product of $\widehat{L}_{\widehat{A}}(\widehat{M})$, not with completed projective tensor product. However, we can write
	 \[
	 \begin{split}
	 \theta i_A  \circ i_1 (x \otimes y) & \stackrel{\ref{compat}}{=} \theta \circ \hat{i}_1 (i_M(x) \otimes i_{M^{-1}}(y) ) \stackrel{\ref{analyt lta}}{=} m \circ (\alpha \hat{\otimes} \beta) ( i_M(x) \otimes  i_{M^{-1}} (y) ) = \\
	 & = m \circ \varphi \circ (\alpha i_M \otimes i_{M^{-1}} \beta) (x \otimes y) = m (\alpha i_M \otimes i_{M^{-1}} \beta) (x \otimes y),
	 \end{split}	 
	 \]
	 where 
	 \[
	 \varphi : \widehat{L}_{\widehat{A}}(\widehat{M}) \otimes_{A} \widehat{L}_{\widehat{A}}(\widehat{M}) \rightarrow \widehat{L}_{\widehat{A}}(\widehat{M}) \hat{\otimes}_{\widehat{A}} \widehat{L}_{\widehat{A}}(\widehat{M}), \quad \varphi(b_1 \otimes b_2) = b_1 \otimes b_2.
	 \] If we denote the algebra $\widehat{L}_{\widehat{A}}(\widehat{M})$ by $B$, then the argument can be illustrated by the following three-dimensional diagram:
	\begin{equation}
		\begin{tikzcd}[row sep=scriptsize, column sep=scriptsize]
			& M \otimes_A M^{-1} \arrow[dl] \arrow[rr, "i_1"] \arrow[dd] & & A \arrow[dl, "i_A"] \arrow[dd] & & M^{-1} \otimes_A M \arrow[dd] \arrow[dl] \arrow[ll, "i_2"] \\
			\widehat{M} \hat{\otimes}_{\widehat{A}} \widehat{M^{-1}} \arrow[rr, crossing over, "\hat{i}_1" near start] \arrow[dd] & & \widehat{A}  & & \widehat{M^{-1}} \hat{\otimes}_{\widehat{A}} \widehat{M} \arrow[ll, crossing over, "\hat{i}_2" near start] \\
			& B \otimes_A B \arrow[dl, "\varphi"] \arrow[rr, "m" near start] & & B \arrow[dl, "Id_B"] & & B \otimes_A B \arrow[dl, "\varphi"] \arrow[ll, "m" near start] \\
			B \hat{\otimes}_{\widehat{A}} B \arrow[rr, "m"] & & B \arrow[from=uu, crossing over]  & & B \hat{\otimes}_{\widehat{A}} B \arrow[from=uu, crossing over] \arrow[ll, "m"]\\
		\end{tikzcd}
	\end{equation}
	And if the triple $(\theta i_A, \alpha i_M, \beta i_{M^{-1}}, \widehat{L}_{\widehat{A}}(\widehat{M}))$ is compatible, we get the morphism $i : L_A(M) \rightarrow \widehat{L}_{\widehat{A}}(\widehat{M})$.
	
	Secondly, we need to prove that the pair $(\widehat{L}_{\widehat{A}}(\widehat{M}), i)$ satisfies the universal property. Without the loss of generality, we can assume that $X$ is a Banach algebra and $\varphi : L_A(M) \rightarrow X$ is an algebra homomorphism. In this case we consider the following continuous morphisms: $\widehat{\left.\varphi \right|_A} : \widehat{A} \rightarrow X$, $\widehat{\left.\varphi \right|_M} : \widehat{M} \rightarrow X$ and $\widehat{\left.\varphi \right|_{M^{-1}}} : \widehat{M^{-1}} \rightarrow X$, which come from the respective universal properties. The first map is an algebra homomorphism, and the latter are $\widehat{A}$-bimodule morphisms. The resulting triple is topologically compatible, and the argument is basically the same as the one we gave in the first step of the proof, we only need to keep in mind that elementary tensors span a dense subspace in a completed projective tensor products of locally convex spaces From that we get a unique $\widehat{A}$-algebra morphism $\hat{\varphi} : \widehat{L}_{\widehat{A}}(\widehat{M}) \rightarrow X$. The last thing that is left is to show that it really extends $\varphi$. However, if we restrict $\varphi$ on $A$, $M$ or $M^{-1}$, the statement holds, so it is true for $L_A(M)$.
\end{proof}

The following is a corollary of Propositions \ref{example where conj is satisfied} and \ref{Arens-Michael env of LTA}.

\begin{corollary}
	\label{firstisom}
	Suppose that $A$ is an associative algebra with the Arens-Michael envelope which is a Fr\'echet algebra, and let $\alpha \in \text{Aut}(A)$ be an arbitrary algebra automorphism. Then the following isomorphism takes place: 
	\[
	(A [ t, t^{-1}; \alpha ])^{\widehat{ }} = \widehat{ L_A(A_\alpha) } \cong \widehat{L}_{\widehat{A}} (\widehat{A}_{\widehat{\alpha}}).
	\]
\end{corollary}

\subsection{Constructing the universal object}
To prove Theorem \ref{topolaurent} we need to utilize the construction of the analytic tensor algebra, described in \cite{PirkAM}.

Suppose that $A$ is an Arens-Michael algebra and $M$ is an $A$-$\hat{\otimes}$-bimodule. Fix a directed generating family of seminorms $\{ \left\|\cdot\right\|_\nu : \nu \in \Lambda \}$ on $M$. Consider the locally convex space
\begin{equation}
\widehat{T}_A(M)^{+} = \left\lbrace  (x_n) \in \prod\limits_{i=1}^{\infty} M^{\hat{\otimes} n} : \left\|(x_n)\right\|_{\nu, \rho} := \sum\limits_{n =1}^{\infty} \left\|x_n\right\|_\nu^{\hat{\otimes} n} \rho^n < \infty, \nu \in \Lambda, 0 < \rho < \infty \right\rbrace. 
\end{equation}

By definition, the seminorms $\left\|\cdot\right\|_{\nu, \rho}$ generate the topology on $\widehat{T}_A(M)^{+}$. 
\begin{definition}
	The topological (analytic) tensor algebra of $M$ is a locally convex space \[\widehat{T}_A(M) := A \oplus \widehat{T}_A(M)^{+}.\]
\end{definition}
In \cite{PirkAM} it is proven that $\widehat{T}_A(M)$ admits a multiplication which makes a natural inclusion ${f : T_A(M) \rightarrow \widehat{T}_A(M)}$ into an algebra homomorphism and turns $\widehat{T}_A(M)$ into an Arens-Michael algebra. We also will use the \cite[Proposition 4.8]{PirkAM}, which states that $\widehat{T}_A(M)$ admits a description via a universal property.

Fix an Fr\'echet-Arens-Michael algebra $A$, a pair of topologically inverse Fr\'echet $A$-$\hat{\otimes}$-bimodules $M$ and $M^{-1}$ with respect to the topological $A$-bimodule isomorphisms $i_1 : M \hat{\otimes}_A M^{-1} \rightarrow A$ and $i_2 : M^{-1} \hat{\otimes}_A M \rightarrow A$.

Now, for any $x \in M$ and $y \in M^{-1}$ consider the elements 
$$(0, 0, (x, 0) \otimes (0, y), 0, \dots) - (i_1(x \otimes y), 0, 0, \dots) \in \widehat{T}_A(M)$$ 
and  
$$(0, 0, (0, y) \otimes (x, 0), 0,  \dots) - (i_2(y \otimes x), 0, 0, \dots) \in \widehat{T}_A(M).$$ It would be reasonable to assume that these elements are equal to zero in $\widehat{L}_A(M)$. This idea serves as a motivation for the following definition:

\begin{definition}
	Let $\widehat{L}_A(M)' := \widehat{T}_A(M \oplus M^{-1}) / I$, where $I$ is the closure of a two-sided ideal generated by $(0, 0, (x, 0) \otimes (0, y), 0, \dots) - (i_1(x \otimes y), 0, 0, \dots)$ and $(0, 0, (0, y) \otimes (x, 0), 0, \dots) - (i_2(y \otimes x), 0, 0, \dots)$ for any $x \in M$, $y \in M^{-1}$.
\end{definition}

\textbf{Remark.} Actually, this is the only place where we use the Fr\'echet assumption. If $\widehat{T}_A(M \oplus M^{-1})$ is not Fr\'echet, the quotient might not be complete, we would have to complete the resulting algebra and the following proof, in fact, will still work, however this assumption makes everything easier.

Let us also denote some morphisms associated with this object:
\[j_0 : A \hookrightarrow \widehat{T}_A(M \oplus M^{-1})\]
\[j_M : M \hookrightarrow M \oplus M^{-1} \rightarrow \widehat{T}_A(M \oplus M^{-1})\]
\[j_{M^{-1}} : M^{-1} \hookrightarrow M \oplus M^{-1} \rightarrow \widehat{T}_A(M \oplus M^{-1})\]
\[\pi : \widehat{T}_A(M \oplus M^{-1}) \rightarrow \widehat{L}_A(M)'\].

If $A$ is an Fr\'echet-Arens-Michael algebra, so are $\widehat{T}_A(M \oplus M^{-1})$ and $\widehat{L}_A(M)'$. Consider a triple of morphisms $i_A = \pi \circ j_0$, $i_M = \pi \circ j_M$, $i_{M^{-1}} = \pi \circ j_{M^{-1}}$.

\begin{lemma}
	The triple $(i_A, i_M, i_{M^{-1}}, \hat{L}_A(M)')$ is topologically compatible.
\end{lemma}
\begin{proof}
	We need to prove the commutativity of the following diagram:
	\begin{equation}
		\begin{tikzcd}
			M \hat{\otimes}_A M^{-1} \arrow[d, "i_M \otimes i_{M^{-1}}"] \ar[r, "i_1"] &
			A \arrow[d, "i_A"] &
			M^{-1} \hat{\otimes}_A M \arrow[d, "i_{M^{-1}} \otimes i_M"] \arrow[l, "i_2"] \\
			\widehat{L}_A(M)' \hat{\otimes}_{A} \widehat{L}_A(M)' \arrow[r, "m"] &
			\widehat{L}_A(M)' &
			\widehat{L}_A(M)' \hat{\otimes}_{A} \widehat{L}_A(M)' \arrow[l, "m"]
		\end{tikzcd}
	\end{equation}
	For every $x \in M$ and $y \in M^{-1}$ we can consider an elementary tensor $x \otimes y \in M \hat{\otimes}_A M^{-1}$.
	\[
		\begin{aligned}
		m \circ (i_M \otimes i_{M^{-1}}) (x \otimes y) &= m((0, (x, 0), 0, 0, \dots) \cdot (0, (0, y), 0, 0, \dots)) = (0, 0, (x, 0) \otimes (0, y), 0, \dots) \\
		& = (i_1(x \otimes y), 0, 0, \dots) = i_A \circ i_1(x \otimes y).
		\end{aligned}
	\]
	By using the fact that elementary tensors span a dense subspace in $M \hat{\otimes}_A M^{-1}$, we finish the proof.
\end{proof}

\begin{proposition}
	$\widehat{L}_A(M)' \simeq \widehat{L}_A(M)$.
\end{proposition}

\begin{proof}
	We must check that the triple $(i_A, i_M, i_{M^{-1}}, \hat{L}_A(M)')$ satisfies the universal property. Suppose that $B$ is a Banach algebra and that  $(\theta, \gamma, \delta, B)$ is an topologically compatible triple. Consider the direct sum of $\gamma$ and $\delta$: $\gamma \oplus \delta : M \oplus M^{-1} \rightarrow B$. It is a continuous $A$-bimodule morphism which, by \cite[Proposition 4.8]{PirkAM}, can be uniquely extended to a continuous $A$-algebra morphism ${\varphi : \widehat{T}_A(M \oplus M^{-1}) \rightarrow B}$. From the fact that $(\theta, \gamma, \delta, B)$ is topologically compatible it easily follows that $\varphi(I) = 0$, so, in fact, we obtain a unique continuous $A$-algebra homomorphism $\tilde{\varphi} : \hat{L}_A(M)' \rightarrow B$. Due to the construction it extends $\theta$, $\gamma$ and $\delta$. 
\end{proof}

This concludes the proof of the Theorem \ref{topolaurent}.

\section{The case of \texorpdfstring{$M = A_\alpha$}{M equals Aalpha}}
\label{special case}

\subsection{Localizable linear maps between locally convex spaces}
\begin{definition}
		Let $A$ be an Arens-Michael algebra and let $\mathcal{F} \subset \mathcal{L}(A)$ be a family of continuous linear maps $A \rightarrow A$. 
		
		Then $F$ is called an $m$-localizable family if the topology on $A$ can be defined a family of submultiplicative seminorms $\{ \left\| \cdot \right\|_\lambda \}_{\lambda \in \Lambda}$, satisfying the following property: for every $T \in \mathcal{F}$ there exists a constant $C_T > 0$, such that
		\[
		\left\| Ta \right\|_\lambda \le C_T  \left\| a \right\|_\lambda \quad \text{for every } a \in A.
		\]
		An operator $T \in \mathcal{L}(E)$ is called ($m$-)localizable $\iff \{T\}$ is a ($m$-)localizable family.
\end{definition}

Suppose now that $A$ is an Arens-Michael algebra, $\alpha$ is a continuous automorphism of $A$, such that $\{\alpha, \alpha^{-1} \}$ is a $m$-localizable family. Fix a generating family of seminorms $\{ \left\|\cdot\right\|_\lambda : \lambda \in \Lambda \}$, then we can define the following vector space:
\begin{equation}
	\mathcal{O}(\mathbb{C}^\times, A) := \left\lbrace  f = \sum_{i=-\infty}^{\infty} a_i t^i : \left\|f\right\|_{\lambda, \rho} := \sum_{-\infty}^{\infty} \left\|a_i\right\|_\lambda \rho^i < \infty \ \forall \lambda \in \Lambda, 0 < \rho < \infty \right\rbrace.
\end{equation}
This vector space with topology, generated by $\left\|\cdot\right\|_{\lambda, \rho}$ becomes a complete locally convex space. Moreover, \cite[Lemma 4.12]{PirkAM} and  in \cite[Proposition 4.15]{PirkAM} state that in our case $\mathcal{O}(\mathbb{C}, A)$ admits a unique multiplication, which is compatible with $\alpha$ (i.e. $ta = \alpha(a)t, t^{-1}a = \alpha^{-1}(a)t^{-1}$ for every $a \in A$) and makes $\mathcal{O}(\mathbb{C}^\times, A)$ into an Arens-Michael algebra, which is denoted by $\mathcal{O}(\mathbb{C}^\times, A; \alpha)$.

\begin{proposition}	
	Under assumptions made above, $\widehat{L}_A(A_\alpha) \simeq \mathcal{O}(\mathbb{C}^\times, A; \alpha)$.
\end{proposition}

\begin{proof}
	Firstly, we must consider natural morphisms 
	\[
	\begin{aligned}
		i_A : A &\hookrightarrow \mathcal{O}(\mathbb{C}^\times, A; \alpha), \\
		i_{A_\alpha} : A_\alpha &\rightarrow \mathcal{O}(\mathbb{C}^\times, A; \alpha), i_{A_\alpha}(1) = t \\
		i_{A_{\alpha^{-1}}} : A_{\alpha^{-1}} &\rightarrow \mathcal{O}(\mathbb{C}^\times, A; \alpha), i_{A_{\alpha^{-1}}}(1) = t^{-1}.
	\end{aligned}
	\]
	We aim to prove that a triple of morphisms $(i_A, i_{A_\alpha}, i_{A_{\alpha^{-1}}}, \mathcal{O}(\mathbb{C}^\times, A; \alpha))$ is an topologically compatible triple, which satisfies universal property. The first part is obviously true due to the construction of $\mathcal{O}(\mathbb{C}^\times, A; \alpha)$.
	
	Suppose that $(\theta, \alpha, \beta, B)$ is another topologically compatible triple. Notice that 
	\[
	{\alpha(1) \beta(1) = \beta(1) \alpha(1) = 1},
	\]
	so $\alpha(1) \in B$ is an invertible element. Then, due to \cite[Proposition 4.14]{PirkAM}, there exists a unique continuous algebra homomorphism $f :\mathcal{O}(\mathbb{C}^\times, A; \alpha) \rightarrow B$, $f(t) = \alpha(1)$. It easily seen that $f i_A = \theta$, $f i_{A_\alpha} = \alpha$, $f i_{A_{\alpha^{-1}}} = \beta$.
\end{proof}

\subsection{The general case}

In this section $A$ is an Arens-Michael algebra and $\alpha$ is an automorphism of $A$. We aim to obtain a description of $\widehat{T}_A({A_\alpha \oplus A_{\alpha^{-1}}})$, similar to the description of $\widehat{T}_A(A_\alpha)$,  obtained in \cite[Proposition 4.9]{PirkAM}. 

For every tuple $w \in W_2$ we denote the $k$-th symbol of $w$ by $w(k)$. Also consider the functions $c_1 : W_2 \rightarrow \mathbb{Z}_{\ge 0}$ and $c_2 : W_2 \rightarrow \mathbb{Z}_{\ge 0}$ which count the number of instances of $1$ and $2$ in a tuple, respectively. Also denote $c(w) = c_1(w) - c_2(w)$. For every element in $W_2$ define an $A$-$\hat{\otimes}$-bimodule as follows:
\begin{enumerate}[label=(\arabic*)]
	\item $A_\emptyset := A$
	\item $A_{(1)} := A_\alpha$, $A_{(2)} := A_{\alpha^{-1}}$
	\item for every $w_1, w_2 \in W_2$ we have $A_{w_1 w_2} := A_{w_1} \hat{\otimes}_A A_{w_2}$
\end{enumerate}

Let $w \in W_2$ be a non-empty element and let $1 \le k \le |w|$. Replace all numbers $2$ in $w$ with $-1$ and denote the new tuple by $w'$. Let us define a function $p(w, k)$ as follows:
\begin{equation}
	p(w, k) = \sum_{i=1}^{k} w'(j) = \sum_{i=1}^{k} 3 - 2 w(j).
\end{equation}

\begin{proposition}
	For every $w \in W_2$ consider a mapping
	\[
	i_w : \prod_{i=1}^{|w|} A_{\alpha^{ w'(i) }} \rightarrow A_{\alpha^n}, \quad  i_w(x_1, \dots x_{|w|}) := x_1 \prod_{i=2}^{|w|} \alpha^{p(w, i-1)}(x_i),
	\]
	where $n = c(w)$. 
	
	Then $i_w$ is a continuous $A$-balanced map which induces a $A$-$\hat{\otimes}$-bimodule isomorphism \\ $i_w : A_w \simeq A_{\alpha^n}$.
\end{proposition}

\begin{proof}
	First of all, let us prove that $i_w$ is a $A$-balanced map:
	\[
	\begin{aligned}
		& i_w(x_1 , \dots, x_i \circ r, x_{i+1}, \dots x_{|w|}) = \\
		&= x_1 \alpha^{p(w, 1)}(x_2) \dots \alpha^{p(w, i-1)} (x_i \alpha^{w'(i)}(r)) \alpha^{p(w, i)}(x_{i+1}) \dots \alpha^{p(w, |w|-1)}(x_{|w|}) = \\ 
		&= x_1 \alpha^{p(w, 1)}(x_2) \dots \alpha^{p(w, i-1)} (x_i) \alpha^{w'(i)+p(w, i-1)}(r) \alpha^{p(w, i)}(x_{i+1}) \dots \alpha^{p(w, |w|-1)}(x_{|w|}).
	\end{aligned}
	\]
	However, by definition, $w'(i)+p(w, i-1) = p(w, i)$, so we get
	\[
	\begin{aligned}
		& x_1 \alpha^{p(w, 1)}(x_2) \dots \alpha^{p(w, i-1)} (x_i) \alpha^{w'(i)+p(w, i-1)}(r) \alpha^{p(w, i)}(x_{i+1}) \dots \alpha^{p(w, |w|-1)}(x_{|w|}) =  \\ & = x_1 \alpha^{p(w_1)}(x_2) \dots \alpha^{p(w, i-1)} (x_i) \alpha^{p(w, i)}(r) \alpha^{p(w, i)}(x_{i+1}) \dots \alpha^{p(w, |w|-1)}(x_{|w|}) = \\ & = i_w(x_1, \dots, x_i, r \circ x_{i+1}, \dots, x_{|w|}).
	\end{aligned}
	\]
	Therefore, $i_w$ is balanced.
	
	Now suppose that $f : \prod_{i=1}^{|w|} A_{\alpha^{ w'(i) }} \rightarrow M$ is a continuous $A$-balanced $A$-$\hat{\otimes}$-bimodule homomorphism. Then we define
	\[
	\tilde{f} : A_{\alpha^n} \rightarrow M, \quad \tilde{f}(a) = f(a, 1, \dots, 1).
	\]
	This map is a well-defined homomorphism of $A$-$\hat{\otimes}$-bimodules: for any $b \in A$ we have
	\[
	\tilde{f}(ba) = f(ba, 1, \dots, 1) = bf(a, 1, \dots, 1) = b \tilde{f}(a),
	\]
	\[
	\begin{aligned}
		\tilde{f}(a) b &= f(a, 1, \dots, 1) \circ b = f(a, 1, \dots, \alpha^{p(w, |w|)}(b)) = \\ &= f(a, 1, \dots, 1, \alpha^{p(w, |w|-1) + p(w, |w|)}(r), 1) = \dots = f(a \alpha^n(b), 1, \dots, 1).
	\end{aligned}
	\]
	We prove that $f = \tilde{f} \circ i_w$ by using a similar argument, which we will omit here.
\end{proof}

\begin{lemma}
	\label{commd}
	The following diagram is commutative:	
	\begin{equation}
		\begin{tikzcd}
			A_{w_1} \hat{\otimes}_A A_{w_2} \arrow[r, equal ] \arrow[d, "i_{w_1} \otimes i_{w_2}"] & A_{w_1w_2} \arrow[d, "i_{w_1w_2}"] \\
			A_{\alpha^{k_1}} \hat{\otimes}_A A_{\alpha^{k_2}} \arrow[r, "\varphi"] & A_{\alpha^{k_1 + k_2}}
		\end{tikzcd},
	\end{equation}

	where $\varphi(a \otimes b) = a \alpha^{k_1}(b)$.
\end{lemma}

\begin{proof}
	Again, it suffices to look at elementary tensors. Let $x = x_1 \otimes \dots \otimes x_{|w_1|} \in A_{w_1}$ and \newline ${y = y_1 \otimes \dots \otimes y_{|w_2|} \in A_{w_2}}$. Then we have	
	\[
	\begin{aligned}
	\varphi \circ (i_{w_1} \otimes i_{w_2}) (x \otimes y) = \varphi & \left(\prod_{i=1}^{|w_1|} \alpha^{p(w_1, i-1)}(x_i) \otimes \prod_{i=1}^{|w_2|} \alpha^{p(w_2, i-1)}(y_i) \right) = \\ &= \prod_{i=1}^{|w_1|} \alpha^{p(w_1, i-1)}(x_i) \cdot \prod_{i=1}^{|w_2|} \alpha^{p(w_2, i-1) + k_1}(y_i).
	\end{aligned}
	\]
	Notice that $k_1 = p(w_1, |w_1|)$, therefore, 
	\[
		\prod_{i=1}^{|w_1|} \alpha^{p(w_1, i-1)}(x_i) \cdot \prod_{i=1}^{|w_2|} \alpha^{p(w_2, i-1) + k_1}(y_i) = \prod_{i=1}^{|w_1|} \alpha^{p(w_1w_2, i-1)}(x_i) \cdot \prod_{i=|w_1|+1}^{|w_2|} \alpha^{p(w_1w_2, i-1)}(y_i) = i_{w_1 w_2} (x \otimes y).
	\]
\end{proof}
Fix a generating family of seminorms $\{ \left\| \cdot \right\|_\lambda : \lambda \in \Lambda \}$ on $A$.
\begin{definition}
	Define the following locally convex space:
	\begin{equation}
	A\{x_1, x_2; \alpha\} = \left\lbrace f = \sum\limits_{w \in W_2} a_w x^w : \left\|f\right\|_{\lambda, \rho} =  \sum\limits_{w \in W_2} \left\|a_w\right\|_\lambda^{(w)}\rho^{|w|} < \infty \  \forall \lambda \in \Lambda, 0 < \rho < \infty  \right\rbrace,
	\end{equation}
	where $\left\|r\right\|_\lambda^{(w)}$ are seminorms on $A$, which are defined as follows:
	\[
	\left\|r\right\|_\lambda^{(w)} = \inf_{r = \sum_{j=1}^k i_w(r_{1, j} \otimes \dots \otimes r_{|w|, j})} \sum_{j=1}^k \left\|r_{1, j}\right\|_\lambda \dots \left\|r_{|w|, j}\right\|_\lambda,
	\]
	and by definition $\left\| \cdot \right\|^{(\emptyset)}_{\lambda} = \left\| \cdot \right\|_\lambda$.
\end{definition}
\noindent
\textbf{Remark.} Actually, $\left\| \cdot \right\|^{(1)}_{\lambda} = \left\| \cdot \right\|^{(2)}_{\lambda} = \left\| \cdot \right\|_\lambda$ due to the definition of $i_w$.

The space $A\{x_1, x_2; \alpha\}$ with the topology, generated by $\left\|\cdot\right\|_{\lambda, \rho}$, is a complete locally convex space.

\begin{theorem}
	\label{thm3.1}
	The space $A\{x_1, x_2; \alpha\}$ admits a unique multiplication which satisfies the following conditions:
	\begin{enumerate}[label=(\arabic*)]
		\item the natural inclusions $A[t; \alpha] \hookrightarrow A\{x_1, x_2; \alpha \}$ and $A[s; \alpha^{-1}] \hookrightarrow A\{x_1, x_2, \alpha \}$, where \newline $\sum a_n t^n \rightarrow \sum a_n x_1^n$ and $\sum a_n s^n \rightarrow \sum a_n x_2^n$, are algebra homomorphisms.
		\item there exists a canonical topological $A$-algebra isomorphism $\psi : A\{x_1, x_2;\alpha\} \rightarrow \widehat{T}_A(A_\alpha \oplus A_{\alpha^{-1}})$.
	\end{enumerate}
	As a corollary, $A\{x_1, x_2; \alpha\}$ becomes an Arens-Michael algebra.
\end{theorem}

\begin{proof}
	Fix a generating directed family of seminorms $\{\left\|\cdot\right\|_\lambda : \lambda \in \Lambda \}$ on $A$. For every $k > 0$ we identify $(A_\alpha \oplus A_{\alpha^{-1}})^{\hat{\otimes} k}$ with $\bigoplus\limits_{|w| = k} A_w$. If we denote the projective tensor product of $k$ copies of $\left\| \cdot \right\|_\lambda + \left\| \cdot \right\|_\lambda$ by $\left\| \cdot \right\|^{\hat{\otimes} n}_{\lambda, \lambda}$, we can rewrite the definition of $\widehat{T}_A(A_\alpha \oplus A_{\alpha^{-1}})$ as follows:
	\[
	\widehat{T}_A(A_\alpha \oplus A_{\alpha^{-1}}) = \left\lbrace  (x_w) \in \prod_{w \in W_2} A_w : \left\| (x_w) \right\|_{\lambda, \rho} =  \sum_{n \ge 0} \left\| (x_w)_{|w| = n} \right\|^{\hat{\otimes} n}_{\lambda, \lambda} \rho^n < \infty, \lambda \in \Lambda, 0 < \rho < \infty \right\rbrace.
	\]
	Moreover, notice that for every $x_w \in A_w$, $\lambda \in \Lambda$ we have 
	\[
	\left\| (x_w)_{|w| = n} \right\|^{\hat{\otimes} n}_{\lambda, \lambda} = \sum_{|w| = n} \left\| i_w(x_w) \right\|^{(w)}_\lambda
	\]
	by the definition of $\left\| \cdot \right\|^{(w)}_{\lambda}$.
	
	For any element $a \in A\{x_1, x_2; \alpha\}$ we define $\psi$ as follows:
	\[
		(\psi(a))_w = a_w \otimes 1 \dots \otimes 1.
	\]
	Therefore, for any $0 < \rho < \infty$ and $\lambda \in \Lambda$ we have
	\[
	\begin{aligned}
		\left\|\psi(f)\right\|_{\lambda, \rho} & = \sum_{n = 0}^\infty \left(  \left\|(a_w \otimes 1 \dots \otimes 1)_{|w| = n} \right\|^{\hat{\otimes} n}_{\lambda} \right) \rho^n = \sum_{n = 0}^\infty \left( \sum_{|w| = n} \left\|i_w(a_w \otimes 1 \dots \otimes 1)\right\|^{(w)}_{\lambda} \right) \rho^n = \\ & = \sum_{n = 0}^\infty \left( \sum_{|w| = n} \left\|a_w\right\|^{(w)}_{\lambda} \right) \rho^n = \left\|f\right\|_{\lambda, \rho}.
	\end{aligned}
	\]
	Therefore, we have proven that $\psi$ is a topological isomorphism of locally convex spaces, and Lemma \ref{commd} ensures that $\psi$ is an algebra homomorphism, and the existence of natural inclusions 
	\[
	T_A(A_\alpha) \hookrightarrow \widehat{T}_A(A_\alpha \oplus A_{\alpha^{-1}}), \quad T_A(A_{{\alpha}^{-1}}) \hookrightarrow \widehat{T}_A(A_\alpha \oplus A_{\alpha^{-1}})
	\] implies (1).
\end{proof}

\begin{corollary}
	\label{secondisom}
	Suppose that $A$ is a Fr\'echet-Arens-Michael algebra and $\alpha$ is an automorphism of $A$. Then 
	\[
	\widehat{L}_A(A_\alpha) \simeq A\{ x_1, x_2; \alpha \} / \overline{(x_1 x_2 - 1, x_2 x_1 - 1)}.
	\]
\end{corollary}

We also provide some examples of explicit computations of $A\{ x, y; \alpha \}$ in the Appendix \ref{appendix A}.

\section{Open questions}
\label{open questions}
\begin{enumerate}
	\item How can we characterize the Arens-Michael envelopes among all Arens-Michael algebras? In particular, is every Arens-Michael algebra isomorphic to the Arens-Michael envelope of an algebra? 
	\item Consider an element $f \in F_2 = \mathbb{C}\left\langle x, y\right\rangle $. Is there a way to determine whether the Arens-Michael envelope of $\widehat{F_2 / (f)}$ is isomorphic to the zero algebra?
	\item Does the Conjecture \ref{conj1} hold for every algebra $A$ and an invertible $R$-bimodule $M$?
	\item There are a lot of interesting algebras for which the Arens-Michael envelopes are yet to be explicitly described. For example, consider the quantum universal enveloping algebra $U_q(\mathfrak{sl}_2)$.
	\begin{definition}
		\label{quantumsl}
		The quantum universal enveloping algebra $U_q(\mathfrak{sl}_2)$ is an associative unital algebra generated by $E$, $F$, $K$, $K^{-1}$ with the following relations:
		\[
		KE = q^2 EK,\ KF = q^{-2} FK,\ [E, F] = \frac{K - K^{-1}}{q - q^{-1}}.
		\]
	\end{definition} 
	When $|q| = 1$, then this algebra can be represented as an iterated Ore extension:
	\[
	U_q(\mathfrak{sl}_2) \simeq \mathbb{C}[K, K^{-1}][F; \alpha_0][E; \alpha_1, \delta],
	\]
	and we have the following result(see \cite{Ped}):
	\begin{theorem}
		Consider $|q| = 1, q \ne -1, 1$.
		Then 
		\[
		\widehat{U_q(\mathfrak{sl}_2)} \simeq \left\lbrace f = \sum\limits_{i, j \in \mathbb{Z}_{\ge 0}, k \in \mathbb{Z}} c_{ijk} K^i F^j E^k : \left\|f\right\|_\rho := \sum\limits_{i,j,k} |c_{ijk}| \rho^{i+j+k} < \infty \ \forall  \rho > 0 \right\rbrace ,
		\] 
		where we endow the space on RHS with multiplication, uniquely defined by the relations in the Definition \ref{quantumsl}.
	\end{theorem}
	When $|q| \ne 1$ this representation becomes useless to us, because the morphisms cease to be $m$-localizable.
	
	In fact, this problem was what motivated the author to tackle the description of the Arens-Michael envelope of Laurent Ore extensions in the general case: consider the following isomorphism:
	\[
	U_q(\mathfrak{sl}_2) \simeq \frac{\mathbb{C}\left< E, F \right>[K, K^{-1}; \alpha] }{([E, F] - \frac{K - K^{-1}}{q - q^{-1}})},
	\]
	where $\alpha(E) = q^2 E$ and $\alpha(F) = q^{-2} F$. Then we use the the main result:
	
	\[
	(\mathbb{C}\left< E, F \right>[K, K^{-1}; \alpha]) \string^ \simeq \widehat{L}_{\mathcal{F}_2} ((\mathcal{F}_2)_\alpha) \simeq \mathcal{F}_2 \{x_1, x_2; \alpha\} / \overline{(x_1 x_2 - 1, x_2 x_1 - 1)}.
	\]
	
	Unfortunately, the algebra $\mathcal{F}_2 \{x, y; \alpha\}$ turned out be too difficult to describe explicitly, the Example \ref{example3} demonstrates the difficulty of the task.
	
\end{enumerate}

\appendix
\section{Several examples of explicit computations of \texorpdfstring{$\widehat{T}_A(A_\alpha)$}{TAalpha} and \texorpdfstring{$\widehat{L}_A(A_\alpha)$}{LAalpha}}
\label{appendix A}
Here we will provide several important examples, which illustrate the complexity of ''extensions'' $\widehat{T}_A(A_\alpha) = A \{ x; \alpha \}, \widehat{T}_A(A_\alpha \oplus A_{\alpha^{-1}}) = A\{x, y; \alpha \}$ (and $\widehat{L}_A(A_\alpha)$ as a corollary) even for the simplest and most natural cases.

We want to consider the case of non-$m$-localizable pairs $\{\alpha, \alpha^{-1} \}$, because the $m$-localizable case has been already treated in Section 3.1.

\begin{lemma}
	\label{themostimportantlemma}
	Consider an Arens-Michael algebra $A$ with topology, generated by a family of seminorms $\{ \left\| \cdot \right\|_\lambda : \lambda \in \Lambda \}$ and $\alpha \in \text{Aut}(A)$. Denote 
	\[
	w_n = \underbrace{(1,1,\dots,1,1,}_{\text{n times}} \underbrace{2,2,\dots,2,2)}_{\text{n times}}, \quad w'_n = \underbrace{(2,2,\dots,2,2,}_{\text{n times}} \underbrace{1,1,\dots,1,1)}_{\text{n times}}.
	\] 
	Suppose that there is an element $r \in A$ such that 
	\begin{equation}
	\label{limit}
		\lim\limits_{n \rightarrow \infty} (\left\| r \right\|^{(w_n)}_\lambda \rho^{2n}) = 0
	\end{equation} 
	for every $\lambda \in \Lambda$ and $\rho > 0$ (in other words, the sequence $(\left\| r \right\|^{(w_n)}_\lambda)$ is rapidly decaying), or 
	\begin{equation}
	\label{limit2}
	\lim\limits_{n \rightarrow \infty} (\left\| r \right\|^{(w'_n)}_\lambda \rho^{2n}) = 0.
	\end{equation}  
	Then $r \in I \subset A \{x_1 , x_2; \alpha \}$, where $I$ is the smallest closed two-sided ideal, which contains $x_1 x_2 - 1$ and $x_2 x_1 - 1$. In particular, if there exists an invertible element $r \in A$, which satisfies \eqref{limit} or \eqref{limit2}, then $\widehat{L}_A(A_\alpha) = 0$. 
\end{lemma}

\begin{proof}
	Notice that $x_1^k x_2^k - 1 \in I$ for any $k > 0$, therefore, $r - r x_1^k x_2^k \in I$, but
	\[
	||(r - r x_1^k x_2^k) - r||_{\lambda, \rho} = \left\| r \right\|^{(w_k)}_\lambda \rho^{2k} \xrightarrow[k \rightarrow \infty]{} 0.
	\]
	As we can see, the sequence $r - r x_1^k x_2^k$ converges to $r$ in the topology of $A \{x_1, x_2; \alpha\}$ due to the assumptions in our Lemma, therefore, $r \in \overline{I}$. 
\end{proof}

\begin{example}
	\label{example1}
	Consider $A = C(\mathbb{R})$ and $\alpha(f)(x) = f(x-1)$ for $f \in C(\mathbb{R}), x \in \mathbb{R}$. Recall that the topology on $A$ is generated by the family $\left\|f\right\|_K := \sup\limits_{x \in K} |f(x)|$, where $K \subset \mathbb{R}$ is a compact subset. Notice that instead of all $K$ we could take all the intervals $[-x, x]$ for $x > 0$ or even $[-a_n, a_n]$, where $(a_n)$ is an arbitrary increasing unbounded sequence. 
	
	Let $|w| > 1$. Then we can write down a lower estimate for $\left\|\cdot\right\|^{(w)}_{[-n, n]}$ as follows:
	\[
	\begin{aligned}
	\left\|f\right\|^{(w)}_{[-n, n]} & = \inf_{f = \sum_{j=1}^k i_w(f_{1, j} \otimes \dots \otimes f_{|w|, j})} \sum_{j=1}^k \left\|f_{1, j}\right\|_{[-n, n]} \dots \left\|f_{|w|, j}\right\|_{[-n, n]} = \\ & = \inf_{f = \sum_{j=1}^k i_w(f_{1, j} \otimes \dots \otimes f_{|w|, j})} \sum_{j=1}^k \left\|f_{1, j}\right\|_{[-n, n]} \dots \left\|\alpha^{p(w, |w|-1)} (f_{|w|, j})\right\|_{[-n+p(w, |w|-1), n+p(w, |w|-1)]} \ge \\ & \ge \inf_{f = \sum_{j=1}^k i_w(f_{1, j} \otimes \dots \otimes f_{|w|, j})} \sum_{j=1}^k \left\|f_{1, j}\right\|_{I_n^{(w)}} \dots \left\|\alpha^{p(w, |w|-1)} (f_{|w|, j})\right\|_{I_n^{(w)}} \ge \left\|f\right\|_{I_n^{(w)}},
	\end{aligned}
	\]
	where \[I_n^{(w)} = \bigcap_{i=1}^{|w|-1} [-n+p(w, i), n+p(w, i)]\]
	for $|w| > 1$, and $I_n^{(w)} = [-n,n]$ for $|w| \le 1$.
	
 	Notice that if the intersection is empty, then we say that the respective seminorm is identically zero.
	
	If we denote 
	\[
	k_{\text{min}}(w) =  \begin{cases}
	\min\limits_{1 \le i \le |w|-1} p(w, i) &, |w| > 1,  \\
	0&, |w| \le 1,
	\end{cases}, \quad k_{\text{max}}(w) = \begin{cases}
	\max\limits_{1 \le i \le |w|-1} p(w, i) &, |w| > 1, \\
	0 &, |w| \le 1,
	\end{cases}
	\]
	then 
	\[
	I_n^{(w)} = [-n + k_{\text{max}}(w), n + k_{\text{min}}(w)].
	\]
	
	Now we aim to prove that $\left\|f\right\|^{(w)}_{[-n, n]} = \left\| f\right\| _{I_n^{(w)}}$. Consider the representation 
	\[
	f = \alpha^{k_{\text{max}}(w)}(g) + \alpha^{k_{\text{min}}(w)}(h) + (f - \alpha^{k_{\text{max}}(w)}(g) - \alpha^{k_{\text{min}}(w)}(h)) \text{ for any } g, h \in C(\mathbb{R}).
	\]
	Thus we get an upper estimate:
	\begin{equation}
	\label{lowestimate1}
	\left\|f\right\|^{(w)}_{[-n, n]} \le \inf_{g, h \in C(\mathbb{R})} \left\|g\right\|_{[-n, n]} + \left\|h\right\|_{[-n, n]} + \left\|f - \alpha^{k_{\text{max}}(w)}(g) - \alpha^{k_{\text{min}}(w)}(h)\right\|_{[-n, n]}
	\end{equation}
	Denote the function $\tilde{f} = f - \alpha^{k_{\text{max}}(w)}(g) - \alpha^{k_{\text{min}}(w)}(h)$. Suppose that $-n + k_{\text{max}} \le n + k_{\text{min}}$. Then consider the following $g_m$ and $h_m$:
	\begin{equation}
	g_m(x) = \begin{cases}
	f(x + k_{\text{max}}), & x < -n - 1/m\\
	(-m(x+n)) f(x + k_{\text{max}}) , & x \in [-n - 1/m, -n]\\
	0, & x > -n
	\end{cases},
	\end{equation}
	
	\begin{equation}
	h_m(x) = \begin{cases}
	0, & x < n \\
	(m(x-n)) f(x + k_{\text{min}}) , & x \in [n, n + 1/m]\\
	f(x + k_{\text{min}}), & x > n + 1/m
	\end{cases}
	\end{equation}
	Then we have to look at the RHS of \eqref{lowestimate1}:
	\[
	\left\|g_m\right\|_{[-n, n]} = \left\|h_m\right\|_{[-n, n]} = 0,
	\]
	\begin{equation}
		\alpha^{k_{\text{max}}(w)}(g_m)(x) = \begin{cases}
		f(x), & x < -n - 1/m + k_{\text{max}}(w)\\
		(-m(x - k_{\text{max}}(w) + n)) f(x) , & x \in [-n - 1/m + k_{\text{max}}(w), -n + k_{\text{max}}(w)]\\
		0, & x > -n + k_{\text{max}}(w)
		\end{cases},
	\end{equation}
	
	\begin{equation}
		\alpha^{k_{\text{min}}(w)}(h_m)(x) = \begin{cases}
			0, & x < n + k_{\text{min}}(w) \\
			(m(x - k_{\text{min}}(w) -n)) f(x) , & x \in [n + k_{\text{min}}(w), n + k_{\text{min}}(w) + 1/m]\\
			f(x), & x > n + 1/m + k_{\text{min}}(w)
		\end{cases}
	\end{equation}

	\begin{equation}
			\tilde{f}(x) = \begin{cases}
		0 & x < -n - 1/m + k_{\text{max}}(w)\\
		(1 + m(x - k_{\text{max}}(w)+n)) f(x) & x \in [-n - 1/m + k_{\text{max}}(w), -n + k_{\text{max}}(w)] \\
		f(x) & x \in [-n + k_{\text{max}}(w), n + k_{\text{min}}(w)] = I_n^{(w)} \\
		(1 - m(x - k_{\text{min}}(w) -n)) f(x) & x \in [n + k_{\text{min}}(w), n + 1/m + k_{\text{min}}(w)]\\
		0 & x > n + 1/m + k_{\text{min}}(w),
		\end{cases}
	\end{equation}
	so 
	\[
		\left\|f\right\|^{(w)}_{[-n, n]} = \left\|\tilde{f}\right\|_{[-n, n]} \le \left\|f\right\|_{[-n+k_{\text{max}}(w) - 1/m, n + 1/m + k_{\text{min}}(w)]}.
	\]
	By taking $m \rightarrow 0$ we get the desired equality.
	
	If $-n + k_{\text{max}}(w) > n + k_{\text{min}}(w)$, then we can look at $g_3(x)$ and $h_3(x)$. Notice that \[-n + k_{\text{max}}(w) - 1/3 > n + k_{\text{min}}(w) + 1/3,\] so the computations above show us that the supports of $\alpha^{k_{\text{max}}(w)}(g_3)$ and $\alpha^{k_{\text{min}}(w)}(h_3)$ have the empty intersection, so $\tilde{f} \equiv 0$ for $g_3$ and $h_3$, therefore 
	\[
	\left\|f\right\|^{(w)}_{[-n, n]} = \left\|\tilde{f}\right\|_{[-n, n]} = 0.
	\]
	
	This argument worked for $|w| > 1$, but we know that
	\[
	\left\| \cdot \right\|^{(w)}_{[-n, n]} = \left\| \cdot \right\|_{[-n, n]}
	\]
	when $|w| \le 1$.

	To sum everything up, we have deduced that the algebras $C(\mathbb{R}) \{ x ; \alpha \}$ and $C(\mathbb{R}) \{ x_1, x_2; \alpha \}$ look as follows:
	\[
		C(\mathbb{R}) \{ x ; \alpha \} = \left\lbrace a = \sum_{k \ge 0} a_k x^k : \left\|a\right\|_{n, \rho} :=\left\|a_0\right\|_{[-n, n]} + \sum_{k \ge 1} \left\|a_k\right\|_{[-n+k-1, n-k+1]} \rho^k < \infty,~ \forall n > 0, 0 < \rho < \infty \right\rbrace,
	\]

	\[		
		C(\mathbb{R}) \{ x_1, x_2; \alpha \} = \left\lbrace  a = \sum_{w \in W_2} a_w x^w : \left\|a\right\|_{n, \rho} := \sum_{w \in W_2}  \left\|a_w\right\|_{I_n^{(w)}} \rho^{|w|} < \infty,~ \forall w \in W_2, 0 < \rho< \infty \right\rbrace .
	\]
	It is easily seen that the isomorphism $C(\mathbb{R}) \{ x ; \alpha \} \simeq C(\mathbb{R})[[x]]$ takes place, because $\left\|\cdot\right\|^{(k)}_{[-n, n]} = 0$ for $k > 2n$. Therefore, Lemma \ref{themostimportantlemma} implies that $\widehat{L}_A(A_\alpha) = 0$, because the $(\left\| 1 \right\|_{[-n, n]}^{(w_k)})_{k \in \mathbb{N}} \in c_{00}$.
\end{example}

\begin{example}
	\label{example2}
What happens if we consider the shift automorphism on the algebra of holomorphic functions $\mathcal{O}(\mathbb{C})$ instead of $C(\mathbb{R})$? We get a result, which is similar to what we got in the Example \ref{example1}, as stated in \cite[Example 4.3]{PirkAM}:
\begin{proposition}
	Let $R = \mathcal{O}(\mathbb{C})$. Consider an automorphism $\alpha(f)(z) = f(z-1)$. Then \newline ${R\{x; \alpha \} \cong R[[x]]}$ as locally convex spaces, where the topology on $R[[x]]$ is generated by $\{ \left\|\cdot\right\|_n \}_{n \in \mathbb{N}}$, where $\left\|\sum_k^\infty a_k x^k\right\|_n = \left\|a_n\right\|$.
\end{proposition}
 The proof cleverly utilizes the Mergelyan's approximation theorem. In particular, A. Yu. Pirkovskii proves that $\left\| \cdot \right\|^{(n+1)}_\rho = 0$ for $n > \lfloor 2\rho \rfloor + 1$, therefore, $\left\| \cdot \right\|^{(w_{n+1})}_\rho \le \left\| \cdot \right\|^{(n+1)}_\rho = 0$, so $\widehat{L}_R(R_\alpha) = 0$, as well. Equivalently, we have
 \[
 (\mathbb{C}[x][y, y^{-1}; \alpha])^{\widehat{}} = 0,
 \]
 where $\alpha(f)(x) = f(x-1)$.
\end{example}

\begin{example}
	\label{example3}
	Let $A = \mathcal{O}(\mathbb{C})$ and consider an automorphism $\beta_q : A \rightarrow A$, $\beta_q(f)(z) = f(qz)$, where $|q| \ne 1$. Fix a generating family of seminorms $\{ \left\|\cdot\right\|_\rho : 0 < \rho < \infty \}$ on $A$, where 
	\[
	\left\|f\right\|_\rho =\left\|  \sum_k f^{(k)} z^k \right\|_\rho := \sum_k |f^{(k)}| \rho^k.
	\]
	
	We want to compute the seminorms $ \left\| \cdot \right\|^{(w)}_\rho $ for a word $w \in W_2$ with $|w| > 1$ and $0 < \lambda < \infty$. Assume that $|q| < 1$. Then for every $f \in A$ we have
	
	\[
	\begin{aligned}
		\left\|f\right\|^{(w)}_\lambda & = \inf_{f = \sum^k_{j=1} i_w(f_{1, j} \otimes \dots \otimes f_{|w|, j})} \sum_{i=1}^k \left\|f_{1, j}\right\|_\lambda \dots \left\|f_{|w|, j}\right\|_\lambda \ge \\ &
		\ge \inf_{f = \sum^k_{j=1} i_w(f_{1, j} \otimes \dots \otimes f_{|w|, j})} \sum_{i=1}^k \left\|f_{1, j}\right\|_{\lambda |q|^{-k_{\text{min}}(w)}} 
		\left\|f_{2, j}\right\|_{\lambda |q|^{p(w, 1)-k_{\text{min}}(w)}}  \dots
		\left\|f_{|w|, j}\right\|_{\lambda |q|^{p(w, |w|-1)-k_{\text{min}}(w)}} \ge \\ & 
		\ge \inf_{f = \sum^k_{j=1} i_w(f_{1, j} \otimes \dots \otimes f_{|w|, j})} \sum_{i=1}^k \left\|f_{1, j}\right\|_{\lambda |q|^{-k_{\text{min}}(w)}} \dots \left\|\beta_q^{p(w, |w|-1)}(f_{|w|, j})\right\|_{\lambda |q|^{-k_{\text{min}}(w)}} \ge \left\|f\right\|_{\lambda |q|^{-k_{\text{min}}(w)}}.
	\end{aligned}
	\]
	However, by considering the representation 
	\[
	f = \beta_q^{k_{\text{min}}(w)}(\beta_q^{-k_{\text{min}}(w)}(f)) = \beta_q^{k_{\text{min}}(w)}(f(q^{-k_{\text{min}}(w)}z))
	\]
	we get
	\[
	\left\| f \right\|^{(w)}_\lambda = \left\| f(q^{-k_{\text{min}}(w)}z) \right\|_\lambda = \left\| f \right\|_{\lambda |q|^{-k_{\text{min}}(w)}}
	\]
	for all $w \in W_2$, and, in conclusion, we have
	\[
	A \{x_1, x_2; \beta_q \} = \left\lbrace f = \sum\limits_{w \in W_2} f_w x^w : \left\|f\right\|_{\lambda, \rho} = \sum\limits_{w \in W_2} (\left\|f_w\right\|_{\lambda |q|^{-k_{\text{min}}(w)}}) \rho^{|w|} < \infty \ \forall 0 < \lambda, \rho < \infty  \right\rbrace.
	\]
	The similar argument in the case when $|q| > 1$ shows that
	\[
	\left\| f \right\|^{(w)}_\lambda = \left\| f \right\|_{\lambda |q|^{-k_{\text{max}}(w)}}
	\]
	for all $w \in W_2$, and
	\[
	A \{x_1, x_2; \beta_q \} = \left\lbrace f = \sum\limits_{w \in W_2} f_w x^w : \left\|f\right\|_{\lambda, \rho} = \sum\limits_{w \in W_2} (\left\|f_w\right\|_{\lambda |q|^{-k_{ \text{max} }(w)}}) \rho^{|w|} < \infty \ \forall 0 < \lambda, \rho < \infty  \right\rbrace.
	\]
	\noindent
	Now, let us instead fix $|q| > 1$. Then, by the definition of $A \{x_1, x_2; \beta_q \}$, we can expand the Taylor series of the coefficients, and write
	\[
	A \{x_1, x_2; \beta_q \} = \left\lbrace f = \sum_{\substack{w \in W_2 \\ m \ge 0}} f^{(m)}_w z^m x^w : \left\| f \right\|_{\lambda, \rho} = \sum_{\substack{w \in W_2 \\ m \ge 0}} |f^{(m)}_w| \lambda^m |q|^{-m k_{ \text{max} }(w) } \rho^{|w|} < \infty \ \forall 0 < \lambda, \rho < \infty \right\rbrace.
	\]
	Let us fix $0 < \lambda, \rho < \infty$. What we are going to do next is to estimate the quotient norms $\left\| f \right\|^{\vee}_{\lambda, \rho}$ with respect to the ideal $I$.
	If we denote
	\[
	w_{k, l} = (\underbrace{(1,1,\dots,1,1,}_{k \text{ times}} \underbrace{2,2,\dots,2,2)}_{l \text{ times}}),
	\]
	then
	\begin{equation}
	\label{maximize}
		k_{\text{max}}(w) \le k_{\text{max}}(w_{c_1(w), c_2(w)}) = \begin{cases}
		c_1(w), & c_2(w) \ne 0, \\
		\max \{c_1(w) - 1, 0\}, & c_2(w) = 0.
		\end{cases}
	\end{equation}
	This will help us to prove the following lemma:
	\begin{lemma}
		\label{first step}
		Let $f \in A \{x_1, x_2; \beta_q \}$. Denote $f^{(m)}_{n_1, n_2} = \sum\limits_{\substack{c_1(w) = n_1 \\ c_2(w) = n_2 }}f^{(m)}_{w}.$ Then
		\begin{equation}
		\label{lemma2firsteq}
			\left\| \sum_{\substack{n_1, n_2 \ge 0 \\ m \ge 0}} f^{(m)}_{n_1, n_2} z^m x_1^{n_1} x_2^{n_2}  \right\|_{\lambda, \rho} \le \left\| f \right\|_{\lambda, \rho}
		\end{equation}
		for every $0 < \lambda, \rho < \infty$. As a corollary, the series belongs to $A \{x_1, x_2; \beta_q \}$.
		
		In particular,
		\begin{equation}
		\label{lemma2seceq}
			\left\| f \right\|^\vee_{\lambda, \rho} \le \left\| \sum_{\substack{n_1, n_2 \ge 0 \\ m \ge 0}} f^{(m)}_{n_1, n_2} z^m x_1^{n_1} x_2^{n_2}  \right\|_{\lambda, \rho}.
		\end{equation}
	\end{lemma}
	\begin{proof}
		Let us just expand the left seminorm by the definition:
		\[
		\begin{aligned}
		\left\| \sum_{\substack{n_1, n_2 \ge 0 \\ m \ge 0}} f^{(m)}_{n_1, n_2} z^m x_1^{n_1} x_2^{n_2}  \right\|_{\lambda, \rho} &= \sum_{\substack{n_1, n_2 \ge 0 \\ m \ge 0}} |f^{(m)}_{n_1, n_2}| \lambda^m |q|^{-m k_{\text{max}}(w_{n_1, n_2})} \rho^{n_1 + n_2} \le \\ 
		&\le \sum_{\substack{n_1, n_2 \ge 0 \\ m \ge 0}} \sum_{\substack{c_1(w) = n_1 \\ c_2(w) = n_2 }}|f^{(m)}_{w}| \lambda^m |q|^{-m k_{\text{max}}(w_{n_1, n_2})} \rho^{n_1 + n_2} \stackrel{\ref{maximize}}{\le} \\ 
		& \stackrel{\ref{maximize}}{\le} \sum_{\substack{n_1, n_2 \ge 0 \\ m \ge 0}} \sum_{\substack{c_1(w) = n_1 \\ c_2(w) = n_2 }}|f^{(m)}_{w}| \lambda^m |q|^{-m k_{\text{max}}(w)} \rho^{|w|} = \left\| f \right\|_{\lambda, \rho}
		\end{aligned}
		\]
	\end{proof}
\noindent
\textbf{Remark.} We denote $x^n = x_1^n$ if $n \ge 0$ and $x^n = x_2^{-n}$ if $n < 0$.
	\begin{lemma}
		\indent
		\begin{enumerate}
			\item If $|q|^m \le \rho$ then $\left\| a z^m x^{n_1 - n_2} \right\|_{\lambda, \rho} \le \left\| a z^m x_1^{n_1} x_2^{n_2} \right\|_{\lambda, \rho}$ for $n_1, n_2 \ge 0$.
			\item If $\rho < |q|^m \le \rho^2$ and $0 \le n_2 < n_1$ then $\left\| a z^m x_1^{n_1 - n_2 + 1} x_2 \right\|_{\lambda, \rho} \le \left\| a z^m x_1^{n_1} x_2^{n_2} \right\|_{\lambda, \rho}$, and if $0 \le n_1 \le n_2$ then $\left\| a z^m x^{n_1 - n_2} \right\|_{\lambda, \rho} \le \left\| a z^m x_1^{n_1} x_2^{n_2} \right\|_{\lambda, \rho}$.
			\item If $\rho^2 < |q|^m$ then $\left\| a z^m x_1^{n_1} x_2^{n_2} \right\|^\vee_{\lambda, \rho} = 0$ for $n_1, n_2 \ge 0$.
		\end{enumerate}
	\end{lemma}
	\begin{proof}
		1. Let us prove this inequality by induction. The base case ($k > 0$):
		\[
		\left\| a z^m x_1^k x_2 \right\|_{\lambda, \rho} = |a| \lambda^m |q|^{-m k} \rho^{k+1} \ge |a| \lambda^m |q|^{-m \max \{ k-2, 0 \} } \rho^{k-1} = \left\| a z^m x_1^{k-1} \right\|_{\lambda, \rho}
		\]
		\[
		\left\| a z^m x_1 x_2^k \right\|_{\lambda, \rho} = |a| \lambda^m |q|^{-m} \rho^{k+1} \ge |a| \lambda^m \rho^{k-1} = \left\| a z^m x_2^{k-1} \right\|_{\lambda, \rho}.
		\]
		The step of induction ($n_1, n_2 > 1$):
		\[
		\left\| a z^m x_1^{n_1} x_2^{n_2} \right\|_{\lambda, \rho} = |a| \lambda^m |q|^{-m n_1} \rho^{n_1 + n_2} \ge |a| \lambda^m |q|^{-m (n_1 - 1)} \rho^{n_1 + n_2 - 2} = \left\| a z^m x_1^{n_1-1} x_2^{n_2-1} \right\|_{\lambda, \rho}.
		\]
		2. Notice that in this case for $k > 0$ we have the following chains of inequalities:
		\[
		\equalto{\left\| a z^m x_1^k \right\|_{\lambda, \rho}}{|a|\lambda^m |q|^{-m(k-1)} \rho^{k}} > \equalto{\left\| a z^m x_1^{k+1} x_2 \right\|_{\lambda, \rho}}{|a|\lambda^m |q|^{-m(k+1)} \rho^{k+2}} \le \equalto{\left\| a z^m x_1^{k+2} x_2^2 \right\|_{\lambda, \rho}}{|a|\lambda^m |q|^{-m(k+2)} \rho^{k+4}} \le \dots
		\]
		\[
		\equalto{\left\| a z^m x_2^{k-1} \right\|_{\lambda, \rho}}{|a|\lambda^m \rho^{k-1} } \le \equalto{\left\| a z^m x_1 x_2^{k} \right\|_{\lambda, \rho}}{|a|\lambda^m |q|^{-m} \rho^{k+1}} \le \equalto{\left\| a z^m x_1^2 x_2^{k+1} \right\|_{\lambda, \rho}}{|a| \lambda^m |q|^{-2m} \rho^{k+3}} \le \dots.
		\]
		3. Suppose that $|q|^m > \rho^2$. Then we consider the following limit:
		\[
		\lim_{l \rightarrow \infty} \left\| a z^m x_1^{n_1 + l} x_2^{n_2 + l} \right\|_{\lambda, \rho} = \lim_{l \rightarrow \infty} |a| \lambda^m |q|^{-m (n_1 + l)} \rho^{n_1 + n_2 + 2l} = \left\| a z^m x_1^{n_1} x_2^{n_2} \right\|_{\lambda, \rho} \lim_{l \rightarrow \infty} \left( \frac{\rho^{2}}{|q|^m}\right)^l = 0.
		\]
	\end{proof}
	As a simple corollary we get the following lemma: 
	\begin{lemma}
		\label{second step}
		Consider $g = \sum\limits_{\substack{n_1, n_2 \ge 0 \\ m \ge 0}} g_{n_1, n_2}^{(m)} z^m x_1^{n_1} x_2^{n_2} \in A \{x_1, x_2; \beta_q \}$. Denote $g^{(m)}_n = \sum\limits_{n_1 - n_2 = n} g^{(m)}_{n_1, n_2}$.
		\begin{enumerate}
			\item 
			\begin{equation}
			\label{lemma4firsteq}
				\left\| \sum_{\substack{n \in \mathbb{Z} \\ |q|^m \le \rho}} g^{(m)}_n z^m x_1^n + \sum_{\substack{n > 0 \\ \rho < |q|^m \le \rho^2}} g^{(m)}_n z^m x_1^{n+1} x_2 + \sum_{\substack{n \le 0 \\ \rho < |q|^m \le \rho^2}} g^{(m)}_n z^m x^n_1 \right\|_{\lambda, \rho} \le \left\| g \right\|_{\lambda, \rho}
			\end{equation}		
			for every $0 < \lambda, \rho < \infty$.
			\item
			\begin{equation}
			\label{lemma4seceq}
				\left\| g \right\|^{\vee}_{\lambda, \rho} \le \left\|  \sum_{\substack{n \in \mathbb{Z} \\ |q|^m \le \rho}} g^{(m)}_n z^m x_1^n + \sum_{\substack{n > 0 \\ \rho < |q|^m \le \rho^2}} g^{(m)}_n z^m x_1^{n+1} x_2 + \sum_{\substack{n \le 0 \\ \rho < |q|^m \le \rho^2}} g^{(m)}_n z^m x^n_1 \right\|_{\lambda, \rho}
			\end{equation}
		\end{enumerate}
	\end{lemma}
	\begin{proposition}
		Suppose that $f \in A \{x_1, x_2; \beta_q \}$. Then
		\begin{enumerate}[label=(\arabic*)]
			\item The sum $\sum\limits_{c(w) = n} f^{(m)}_w$ converges for every $m \ge 0$ and $n \in \mathbb{Z}$. Moreover, the linear functionals $\varphi_{m, n} : f \rightarrow \sum\limits_{c(w) = n} f^{(m)}_w$ are continuous.
			\item If these sums are equal to zero for every $m \ge 0$ and $n \in \mathbb{Z}$, then $f \in I$. In other words, $\bigcap_{m, n} \text{Ker}(\varphi_{m, n}) \subset I$.
			\item $I = \bigcap_{m, n} \text{Ker}(\varphi_{m, n})$.
		\end{enumerate}
	\end{proposition}
	\begin{proof}
		(1) Choose $\rho \in \mathbb{R}$ such that $|q|^m < \rho$. Then
		\[
		|\varphi_{m, n}(f)| \le \sum\limits_{c(w) = n} |f^{(m)}_w| < \sum\limits_{c(w) = n} |f^{(m)}_w| |q|^{-m k_{\text{max}}(w)} \rho^{|w|} \le ||f||_{1, \rho}.
		\]
		We used the obvious inequality $k_{\text{max}}(w) \le |w|$.\\
		(2) Suppose that $f \in \bigcap_{m, n} \text{Ker}(\varphi_{m, n})$. Then due to the Lemma \ref{second step} we have 
		\[
		\left\| f \right\|^{\vee}_{\lambda, \rho} \le \left\| \sum_{\substack{n \in \mathbb{Z} \\ |q|^m \le \rho^2}} f^{(m)}_n z^m x^n \right\|_{\lambda, \rho} = \left\| \sum_{\substack{n \in \mathbb{Z} \\ |q|^m \le \rho^2}} \varphi_{m, n}(f) z^m x^n \right\|_{\lambda, \rho} = 0
		\]
		for every $0 < \lambda, \rho < \infty$. Therefore, $f \in I$.\\
		(3) We will show that $\text{Ker}(\lambda_{m, n})$ is a two-sided ideal. Consider $f \in \bigcap_{m, n} \text{Ker}(\lambda_{m, n})$ and $g \in A \{x_1, x_2; \beta_q \}$. Then
		\[
		\begin{aligned}
		fg = \left( \sum_{m_1, w_1} f^{(m_1)}_{w_1} z^{m_1} x^{w_1} \right) \left( \sum_{m_2, w_2} g^{(m_2)}_{w_2} z^{m_2} x^{w_2} \right) = \sum_{\substack{m_1, w_1 \\ m_2, w_2}} f^{(m_1)}_{w_1} g^{(m_2)}_{w_2} q^{c(w_1)} z^{m_1 + m_2}x^{w_1 w_2}.
		\end{aligned}
		\]
		Therefore, $(fg)^{(m)}_w = \sum\limits_{\substack{m_1 + m_2 = m \\ w_1 w_2 = w }} q^{c(w_1)} f^{(m_1)}_{w_1} g^{(m_2)}_{w_2}$, and
		\[
		\begin{aligned}
		\sum_{c(w) = n} (fg)^{(m)}_w = & \sum_{\substack{m_1 + m_2 = m \\ c(w_1) + c(w_2) = n }} q^{c(w_1)} f^{(m_1)}_{w_1} g^{(m_2)}_{w_2} = \sum_{\substack{m_1 + m_2 = m \\ c(w_1) + c(w_2) = n }} q^{n - c(w_2)} g^{(m_2)}_{w_2} f^{(m_1)}_{w_1} = \\ & = \sum_{\substack{0 \le m_2 \le m \\ w_2 \in W_2}} q^{n - c(w_2)} g_{w_2}^{(m_2)} \left( \sum_{c(w_1) = n - c(w_2)} f^{(m - m_2)}_{w_1} \right)  = 0.
		\end{aligned}
		\]
		Similarly, $\bigcap_{m, n} \text{Ker}(\lambda_{m, n})$ can be shown to be a left ideal. Obviously, $x_1 x_2 - 1 \in \bigcap_{m, n} \text{Ker}(\varphi_{m, n})$ and $x_2 x_1 - 1 \in \bigcap_{m, n} \text{Ker}(\varphi_{m, n})$, therefore $I \subset \bigcap_{m, n} \text{Ker}(\varphi_{m, n})$.
	\end{proof}
	Let us denote $\varphi_{m, n}(f) := f^{(m)}_n$.
	
	This characterization of the ideal $I$ allows us to explicitly compute $\left\| f \right\|^\vee_{\lambda, \rho}$ for any $f \in A \{x_1, x_2; \beta_q \}$.
	\begin{lemma}
		For every $f \in A\{ x_1, x_2 ; \beta_q \}$, and $0 < \lambda, \rho < \infty$ we have the following equality:
		\[
		\left\| f \right\|^\vee_{\lambda, \rho} = \left\| \sum_{\substack{n \in \mathbb{Z} \\ |q|^m \le \rho}} f^{(m)}_n z^m x^n + \sum_{\substack{n > 0 \\ \rho < |q|^m \le \rho^2}} f^{(m)}_{n-1} z^m x_1^n x_2 + \sum_{\substack{n \le 0 \\ \rho < |q|^m \le \rho^2}} f^{(m)}_{n} z^m x_2^n \right\|_{\lambda, \rho}.
		\]
	\end{lemma} 
	\begin{proof}
	We already know the inequality in one direction, now let $g \in I$.
	\[
	\begin{aligned}
		& \left\| f + g \right\|_{\lambda, \rho} \stackrel{\ref{first step}}{\ge} \left\| \sum_{\substack{n_1, n_2 \ge 0 \\ m \ge 0}}  (f^{(m)}_{n_1, n_2} + g^{(m)}_{n_1, n_2}) z^m x_1^{n_1} x_2^{n_2} \right\|_{\lambda, \rho} \stackrel{\ref{second step}}{\ge}
		\\ &  \stackrel{\ref{second step}}{\ge}  \left\| \sum_{\substack{n \in \mathbb{Z} \\ |q|^m \le \rho }}  (f^{(m)}_n + g^{(m)}_n) z^m x^n  + \sum_{\substack{n > 0 \\ \rho < |q|^m \le \rho^2}} (f^{(m)}_{n-1} + g^{(m)}_{n-1}) z^m x_1^n x_2  + \sum_{\substack{n \le 0 \\ \rho < |q|^m \le \rho^2}} (f^{(m)}_n + g^{(m)}_n) z^m x_2^n \right\|_{\lambda, \rho} = 
		\\ & = \left\| \sum_{\substack{n \in \mathbb{Z} \\ |q|^m \le \rho }}  f^{(m)}_n z^m x^n + \sum_{\substack{n > 0 \\ \rho < |q|^m \le \rho^2}} f^{(m)}_{n-1} z^m x_1^n x_2  \sum_{\substack{n \le 0 \\ \rho < |q|^m \le \rho^2}} f^{(m)}_n z^m x_2^n \right\|_{\lambda, \rho}.
	\end{aligned}
	\] 
	\end{proof}
	\begin{proposition}
		\[
		\widehat{L}_A(A_\alpha) \simeq \left\lbrace f = \sum_{\substack{m \ge 0 \\ n \in \mathbb{Z}}} f^{(m)}_n z^m x^n : \left\| f \right\|'_{\lambda, \rho} < \infty \ \forall 0 < \lambda, \rho < \infty   \right\rbrace,
		\]
		where 
		\[\left\|f \right\|'_{\lambda, \rho} = \sum_{\substack{n \in \mathbb{Z} \\ |q|^m \le \rho}} |f^{(m)}_n| \lambda^m |q|^{-m \max\{n - 1, 0\} }  \rho^n + \sum_{\substack{n > 0 \\ \rho < |q|^m \le \rho^2}} |f^{(m)}_{n-1}| \lambda^m |q|^{-m (n-1)} \rho^{n+1} + \sum_{\substack{n \le 0 \\ \rho < |q|^m \le \rho^2}} |f^{(m)}_{n}| \lambda^m \rho^n.
		\]
	\end{proposition}
\end{example}

\printbibliography
\Addresses
\end{document}